\documentclass[11pt]{article}
\usepackage{srcltx}
\usepackage{eurosym}
\usepackage{mathtools}
\usepackage{amsmath}
\usepackage{amsfonts}
\usepackage{amssymb}
\usepackage{amsthm}
\usepackage{graphicx}
\usepackage{mathrsfs}
\usepackage{xcolor}
\usepackage{exscale}
\usepackage{color, soul}
\usepackage{latexsym}
\usepackage{authblk}
\usepackage{comment}
\allowdisplaybreaks
\usepackage[colorlinks,plainpages=true,pdfpagelabels,hypertexnames=true,colorlinks=true,pdfstartview=FitV,linkcolor=blue,citecolor=red,urlcolor=black]{hyperref}
\PassOptionsToPackage{unicode}{hyperref}
\PassOptionsToPackage{naturalnames}{hyperref}
\usepackage{enumerate}
\usepackage[shortlabels]{enumitem}
\usepackage{bookmark}
\usepackage{wasysym}
\usepackage{esint}
\newcommand{\assref}{\hyperref[assumption]{\(\mathcal{A}\)}}

\usepackage[ddmmyyyy]{datetime}
\usepackage[margin=1in]{geometry}
\makeatletter
\g@addto@macro\normalsize{%
	\setlength\abovedisplayskip{4pt}
	\setlength\belowdisplayskip{4pt}
	\setlength\abovedisplayshortskip{4pt}
	\setlength\belowdisplayshortskip{4pt}
}
\numberwithin{equation}{section}
\everymath{\displaystyle}
\usepackage[capitalize,nameinlink]{cleveref}
\crefname{section}{Section}{Sections}
\crefname{subsection}{Subsection}{Subsections}
\crefname{condition}{Condition}{Conditions}
\crefname{hypothesis}{Hypothesis}{Conditions}
\crefname{assumption}{Assumption}{Assumptions}
\crefname{lemma}{Lemma}{Lemmas}
\crefname{definition}{Definition}{Definitions}

\crefformat{equation}{\textup{#2(#1)#3}}
\crefrangeformat{equation}{\textup{#3(#1)#4--#5(#2)#6}}
\crefmultiformat{equation}{\textup{#2(#1)#3}}{ and \textup{#2(#1)#3}}
{, \textup{#2(#1)#3}}{, and \textup{#2(#1)#3}}
\crefrangemultiformat{equation}{\textup{#3(#1)#4--#5(#2)#6}}%
{ and \textup{#3(#1)#4--#5(#2)#6}}{, \textup{#3(#1)#4--#5(#2)#6}}%
{, and \textup{#3(#1)#4--#5(#2)#6}}

\Crefformat{equation}{#2Equation~\textup{(#1)}#3}
\Crefrangeformat{equation}{Equations~\textup{#3(#1)#4--#5(#2)#6}}
\Crefmultiformat{equation}{Equations~\textup{#2(#1)#3}}{ and \textup{#2(#1)#3}}
{, \textup{#2(#1)#3}}{, and \textup{#2(#1)#3}}
\Crefrangemultiformat{equation}{Equations~\textup{#3(#1)#4--#5(#2)#6}}%
{ and \textup{#3(#1)#4--#5(#2)#6}}{, \textup{#3(#1)#4--#5(#2)#6}}%
{, and \textup{#3(#1)#4--#5(#2)#6}}

\crefdefaultlabelformat{#2\textup{#1}#3}
\newtheorem{theorem} {Theorem}[section]

\theoremstyle{plain}
\newtheorem{remark}[theorem]{Remark}
\newtheorem*{assumption}{Assumption ($\mathcal{A}$)}
\newtheorem{lemma}[theorem]{Lemma}

\theoremstyle{remark}
\newtheorem{example}{Example}
\newtheorem{counter example}[theorem]{Counter Example}
\def\CC{{\rm \kern.24em \vrule width.02em height1.4ex depth-.05ex \kern-.26emC}}

\def\TagOnRight

\def\AA{{it I} \hskip-3pt{\tt A}}

\def\QQ{\rlap {\raise 0.4ex \hbox{$\scriptscriptstyle |$}} {\hskip -0.1em Q}}

\makeatletter
\newcommand{\vo}{\vec{o}\@ifnextchar{^}{\,}{}}
\makeatother
\def\YYint#1#2#3{{\setbox0=\hbox{$#1{#2#3}{\iint}$}
		\vcenter{\hbox{$#2#3$}}\kern-.50\wd0}}
\def\XXint#1#2#3{{\setbox0=\hbox{$#1{#2#3}{\int}$}
		\vcenter{\hbox{$#2#3$}}\kern-.50\wd0}}

\makeatletter
\def\namedlabel#1#2{\begingroup
	\def\@currentlabel{#2}%
	\label{#1}\endgroup
}
\makeatother
\makeatletter
\newcommand{\rmh}[1]{\mathpalette{\raisem@th{#1}}}
\newcommand{\raisem@th}[3]{\hspace*{-1pt}\raisebox{#1}{$#2#3$}}
\makeatother

\newcounter{desccount}

\newcommand{\descref}[2]{\hyperref[#1]{\textnormal{\textcolor{black}{}\textcolor{blue}{ #2}\textcolor{black}{}}}}
\newcommand{\dref}[2]{\hyperref[#1]{\textcolor{black}{(}\textcolor{blue}{\bf #2}\textcolor{black}{)}}}
\newcommand{\be} {\begin{eqnarray}}
	\newcommand{\ee} {\end{eqnarray}}
\newcommand{\Bea} {\begin{eqnarray*}}
	\newcommand{\Eea} {\end{eqnarray*}}

\newcounter{whitney}
\refstepcounter{whitney}

\newcounter{ineqcounter}
\refstepcounter{ineqcounter}
\makeatletter
\def\ps@pprintTitle{%
	\let\@oddhead\@empty
	\let\@evenhead\@empty
	\def\@oddfoot{}%
	\let\@evenfoot\@oddfoot}
\makeatother
\usepackage[doublespacing]{setspace}
\usepackage[titletoc,toc,page]{appendix}

\makeatletter
\newcommand{\refcheckize}[1]{%
	\expandafter\let\csname @@\string#1\endcsname#1%
	\expandafter\DeclareRobustCommand\csname relax\string#1\endcsname[1]{%
		\csname @@\string#1\endcsname{##1}\wrtusdrf{##1}}%
	\expandafter\let\expandafter#1\csname relax\string#1\endcsname
}
\makeatother
\refcheckize{\cref}
\refcheckize{\Cref}


\makeatletter
\newcommand{\mainsectionstyle}{%
	\renewcommand{\@secnumfont}{\bfseries}
	\renewcommand\section{\@startsection{section}{2}%
		\z@{.5\linespacing\@plus.7\linespacing}{-.5em}%
		{\normalfont\bfseries}}%
}
\makeatother
\usepackage{pgf,tikz}
\usetikzlibrary{arrows}
\usetikzlibrary{decorations.pathreplacing}
\usepackage{enumitem}    % for custom list labels

\usepackage{xpatch}
\xpatchcmd{\MaketitleBox}{\hrule}{}{}{}% remove first horizontal rule (above abstract)
\xpatchcmd{\MaketitleBox}{\hrule}{}{}{}% remoce second horizonral rule (below keywords)

\date{}
\usepackage{scalerel}
\makeatletter

\usepackage{subcaption}
\usepackage{hyperref}
\usepackage{caption}
\usepackage{subcaption}

\usepackage[utf8]{inputenc}
\allowdisplaybreaks

\DeclareCaptionSubType*[Alph]{figure}
\usepackage{hyperref} % optional, for clickable emails

\title%
{Existence and stability of the Riemann solutions for a non-symmetric Keyfitz--Kranzer type model}
\author[1]{Rahul Barthwal\thanks{\href{mailto:rahul.barthwal@mathematik.uni-stuttgart.de}{Corresponding author: rahul.barthwal@mathematik.uni-stuttgart.de}}}
\author[1]{Christian Rohde\thanks{\href{mailto:christian.rohde@mathematik.uni-stuttgart.de}{christian.rohde@mathematik.uni-stuttgart.de}}}
\author[2]{Anupam Sen\thanks{\href{mailto:anupam.sen@vitbhopal.ac.in}{anupamsen@vitbhopal.ac.in}}}

\affil[1]{\footnotesize Institute of Applied Analysis and Numerical Simulation, University of Stuttgart\\
Pffafenwaldring 57, Stuttgart, 70569, Germany}
\affil[2]{\footnotesize Mathematics Division, School of Advanced Sciences and Languages, VIT Bhopal University, Sehore,
Madhya Pradesh, 466114, India}
\begin{document}
\maketitle
\begin{abstract}
In this article, we develop a new hyperbolic model governing the first-order dynamics of a thin film flow under the influence of gravity and solute transport. The obtained system turns out to be a non-symmetric Keyfitz-Kranzer type system. We find an entire class of convex entropies in the regions where the system remains strictly hyperbolic. By including delta shocks, we prove the existence of unique solutions of the Riemann problem. We analyze their stability with respect to the perturbation of the initial data and to the gravity and surface tension parameters. Moreover, we discuss the large time behaviour of the solutions of the perturbed Riemann problem and prove that the initial Riemann states govern it. Thus, we confirm the structural stability of the Riemann solutions under the perturbation of initial data. Finally, we validate our analytical results with well-established numerical schemes for this new system of conservation laws.
\end{abstract}
\medskip \\
{\textbf{MSC codes.}  35L40 35L45 35L65 76A20 
\section{Introduction}
Keyfitz-Kranzer type systems \cite{keyfitz1980system} are an important class of hyperbolic systems of conservation laws given by
\begin{align}\label{eq: KK}
u_{i_{t}}+\big(u_{i}(\phi(u_1, u_2, \ldots, u_n))\big)_{x}=0, ~i=1, 2, \ldots, n, ~n\in \mathbb{N}.
\end{align}
In recent years, the system \eqref{eq: KK} has been analyzed extensively for different choices of $\phi$. Some of the particular choices of $\phi$ include $\phi(u)=\phi(|u|)$ \cite{ freistuhler1991rotational, freistuhler1994cauchy} for $(n\times n)$ systems while $\phi(u_1, u_2)=\phi(\alpha u_1+\beta u_2)$, $\phi(u_1, u_2)=\phi(u_1/u_2)$ \cite{yang2012new, yang2014delta} and $\phi(u_1, u_2)=\phi(u_1u_2)$ \cite{shen2018delta} for two-component systems. Actually, the different choices of $\phi$ describe quite distinct physical phenomena. For instance, the choice of $\phi(u)=\phi(|u|)$ describes a simplified magnetohydrodynamics model \cite{MR2205154, freistuhler1991rotational} or the elastic string model \cite{keyfitz1980system}, while $\phi(u)=k_i u_i/(1+\sum u_i)$ describes multi-component chromatography processes \cite{james1995kinetic}. 

For the particular $(2\times 2)$ case with $\phi(u)=f(u_1)$, Karlsen et al. \cite{karlsen2008semi} developed a semi-Godunov scheme based on Riemann solutions for general triangular systems, including the system of the form
\begin{equation}\label{pressureless}
\begin{aligned}
    u_{1,t}+(u_1f(u_1))_{x}&=0,\\
    u_{2,t}+(u_2f(u_1))_{x}&=0.
  \end{aligned}  
\end{equation}
The system \eqref{pressureless} represents the simplest $(2\times 2)$ triangular system of conservation laws.

Another important Keyfitz-Kranzer type system was recently introduced in Conn et al. \cite{conn2017simple}. It takes the form 
\begin{equation}\label{thin-film}
\begin{aligned}
    h_t+(h^2b/2)_x&=0,\\
    b_t+(hb^2/2)_x&=0.
    \end{aligned}
\end{equation}
Clearly, the system \eqref{thin-film} is a Keyfitz-Kranzer type system with $\phi(h, b)=hb/2$. The system \eqref{thin-film} governs the first-order dynamics of a surface tension-driven thin film flow with an anti-surfactant solute in the absence of gravity. The system \eqref{thin-film} has been analyzed by Sen\&Raja Sekhar and Minhajul et al. for classical and nonclassical wave interactions in 1-D \cite{sekhar2019stability, sen2020delta} while Barthwal et al. and Pandey et al. \cite{barthwal2022two, barthwal2023construction, pandey2025construction} considered multi-D cases.

A central question in the theory of Cauchy problems for \eqref{eq: KK} concerns the stability of Riemann solutions under small perturbations of the initial data or flux parameters. Perturbed Riemann problems involving nonlinear wave interactions have been studied extensively to establish the stability of these solutions (see, e.g., \cite{guo2014perturbed, li2023perturbed, shen2010stability, zhang2016interactions, zhang2025stability} and references therein). We also refer the interested readers to a recent work \cite{tan2025generic} for structural stability of Riemann solutions with flux perturbations.

Recent progress in this area leads us to study the structure of the system \eqref{thin-film}, when we include the gravity effects. Is the new system well-posed and has unique and stable Riemann solutions, especially in the domains where the system becomes weakly hyperbolic? Furthermore, we analyze the effects of the gravity and surface tension parameters and perturbation of initial data on the stability of these solutions.

In this article, we develop a mathematical model which governs the first-order dynamics of a thin film flow under the influence of gravity and anti-surfactant solute. It turns out to be a non-symmetric Keyfitz-Kranzer type system \eqref{eq: KK} with $\phi(u_1, u_2)=\alpha u_1u_2+\kappa u_1^2/3$ for $\alpha, \kappa\geq 0$. The model obtained in this article extends the mathematical model considered in \cite{barthwal2023construction, conn2016fluid} by considering gravity effects. Precisely, we consider the following system
\begin{equation}\label{eq: main_system}
\begin{aligned}
    h_t+(\alpha h^2b+\kappa h^3/3)_x&=0,\\
    b_t+(\alpha hb^2+\kappa h^2b/3)_x&=0,
\end{aligned}    
\end{equation}
where $h(x, t)$ is the dimensionless film thickness and $b(x, t)$ is the concentration gradient and $\alpha$ and $\kappa$ are two nonnegative parameters defining the surface tension and gravity effects. Further, the state space is defined as
\begin{align}\label{state_space}
    \mathcal{U}=\{(h, b)^\top\in \mathbb{R}^2|~~ h, b\geq 0\}.
\end{align}
\begin{remark}
In this article, we focus on the state space \eqref{state_space}, where both variables are non-negative. This, in particular, includes possible rupture of the thin film. The case, when $h> 0$ and $b< 0$ can also be analyzed in a similar manner, where only the ordering of eigenvalues is different.
\end{remark}

In the second part of the article, we consider the local well-posedness of the Cauchy problem with initial data $(h_0(x), b_0(x))^\top\in \mathcal{U}^{\circ}$, where $\mathcal{U}^{\circ}$ denotes the interior of $\mathcal{U}$. Using explicit Riemann invariants, we first develop a class of strictly convex entropies for the system \eqref{eq: main_system} in the regions where the system remains strictly hyperbolic. This then proves the local well-posedness of the smooth solutions. Moreover, we consider a Riemann problem for the system \eqref{eq: main_system} with the following initial data
\begin{align}\label{Riemann_data}
    (h_0, b_0)^\top(x)=\begin{cases}
        (h_-, b_-)^\top, \quad {\rm{if}}~ x<0,\\
        (h_+, b_+)^\top, \quad {\rm{if}}~x>0,
    \end{cases}
\end{align}
where the Riemann data in \eqref{Riemann_data} satisfies
\begin{assumption}\label{assumption}
Only one of the four states $h_-, h_+, b_-$ or $b_+$ can be zero.
\end{assumption}
Under the assumption (\assref), we explicitly compute the intermediate states of the Riemann problem for various choice of initial data and prove that the Riemann problem \eqref{eq: main_system}-\eqref{Riemann_data} has a unique solution in the state space. For $\alpha=1/2$ and $\kappa=0$, the system reduces to the thin film flow system \eqref{thin-film} without gravity. We analyze the behaviour of Riemann solutions for \eqref{eq: main_system} as $\kappa\rightarrow 0$, which can be considered as a vanishing gravity limit for the system \eqref{thin-film}. We explicitly prove that the Riemann solutions for \eqref{eq: main_system}-\eqref{Riemann_data} converge to the Riemann solutions of \eqref{thin-film}-\eqref{Riemann_data} as $\kappa\rightarrow 0$. Similarly, for the case $\alpha=0$ and $\kappa=1$, the system reduces to the triangular system \eqref{pressureless} with $f(h)=h^2/3$. We analyze the vanishing surface tension limit ($\alpha\rightarrow 0$) and prove that the solution of the Riemann problem for \eqref{eq: main_system}-\eqref{Riemann_data} converges to the solution of the Riemann problem for \eqref{pressureless}. 

Furthermore, we consider the sensitivity of the developed Riemann solutions when the initial Riemann data \eqref{Riemann_data} is perturbed. In particular, we consider a Cauchy problem for \eqref{eq: main_system} with the initial data of the form
\begin{align}\label{initial_data}
    (h_0, b_0)^\top(x)=\begin{cases}
        (h_-, b_-)^\top, &{\rm{if}}~ -\infty<x<-\varepsilon,\\
        (h_m, b_m)^\top, & {\rm{if}}~ -\varepsilon<x<\varepsilon,\\
        (h_+, b_+)^\top, &{\rm{if}}~~\varepsilon<x<+\infty,
    \end{cases}
\end{align}
where $\varepsilon>0$ is small parameter and the state $(h_m, b_m)^\top\in \mathcal{U}$ is independent of $\varepsilon$. Solving the Cauchy problem \eqref{eq: main_system}-\eqref{initial_data} is not that straightforward as it involves the interaction between classical and non-classical nonlinear waves, including delta shocks. We adopt a similar method used in \cite{shen2010stability} to obtain the precise point(s), time(s) and curve(s) of interaction, which depend on the parameter $\varepsilon$. We also obtain the strength of the measure-valued delta shock solution for all different cases. By taking the limit $\varepsilon\rightarrow 0$, we prove case by case that the solution of the Cauchy problem \eqref{eq: main_system}-\eqref{initial_data} converges to the solution of \eqref{eq: main_system}-\eqref{Riemann_data}. Also, we prove that the asymptotic behaviour ($t\rightarrow \infty$) of the perturbed Riemann problem is governed by unperturbed left and right Riemann states.

This, in particular, implies that the solution of the Riemann problem for the system \eqref{eq: main_system} remains stable under the perturbation of physical parameters and initial data. We also validate these results by using the Godunov scheme based on the constructed Riemann solutions, which has been established as a prototype stable and convergent first-order finite volume scheme for many hyperbolic systems of conservation laws. We also use a finite-difference Eulerian-Lagrangian scheme developed in \cite{abreu2019fast, de2021interaction} for some examples, which include delta shocks.

We now summarize the key results of this article. For all different cases of initial data \eqref{initial_data} in $\mathcal{U}$ and parameter choices $\alpha$ and $\kappa$, we prove the following.
\begin{theorem}[Local wellposedness of smooth solutions in $\mathcal{U}^{\circ}$]\label{local_wellposedness}
Let the initial data \((h_0, b_0)\) be given such that 
\[
(h_0, b_0) \in (H^m(\mathbb{R}))^2, \quad m > \tfrac{3}{2}, 
\qquad h_0 > 0, \; b_0 > 0.
\]
Then, there exists a finite time \(T^* \in (0, \infty)\) for which the Cauchy problem 
associated with \eqref{eq: main_system} admits a unique solution. Moreover, the 
solution satisfies
\[
(h, b)^\top \in C\!\left([0, T^*]; (H^m(\mathbb{R}))^2\right) 
\cap C^1\!\left([0, T^*]; (H^{m-1}(\mathbb{R}))^2\right).
\]
\end{theorem}
\begin{comment}
\begin{remark}
Note that the region $\mathcal{U}$ is an invariant region for the system \eqref{eq: main_system} using the results from Hopf \cite{} (Corollary 3.3 in \cite{}). Thus, for positive initial data    
\end{remark}
\end{comment}
\begin{theorem}[Existence of Riemann solutions in $\mathcal{U}$]\label{theorem:Riemann}
For any given left state $(h_-, b_-)^\top\in \mathcal{U}$ and the right state $(h_+, b_+)^\top\in \mathcal{U}$ satisfying the assumption (\assref), the Riemann problem \eqref{eq: main_system}-\eqref{Riemann_data} admits a unique solution in $\mathcal{U}$.
\end{theorem}
\begin{theorem}[Vanishing gravity and surface tension limit]\label{vanishing_theorem}\hspace{8 cm}\\
For every $T>0$,
\begin{enumerate}
\item Let $\alpha=1/2$. When the gravity parameter $\kappa\rightarrow 0$, the solution of the Riemann problem  \eqref{eq: main_system}-\eqref{Riemann_data} converges to the solution of the Riemann problem  \eqref{thin-film}-\eqref{Riemann_data} in the space of Radon measures $\mathcal{M}_{Loc}(\mathbb{R};\mathbb{R}^2)$. 
\item Let $\kappa=1$. When the surface tension parameter $\alpha\rightarrow 0$, the solution of the Riemann problem \eqref{eq: main_system}-\eqref{Riemann_data} converges to the solution of the Riemann problem  \eqref{pressureless}-\eqref{Riemann_data} with $f(h)=h^2/3$ in the space of  Radon measures $\mathcal{M}_{Loc}(\mathbb{R};\mathbb{R}^2)$. 
\end{enumerate}
\end{theorem}
\begin{remark}
In the case of classical Riemann solutions, the limit solutions in Theorem \ref{vanishing_theorem} coincide with their $L^1$-limit.    
\end{remark}
\begin{theorem}[Stability of the Riemann solutions with perturbed initial data]\label{perturbed_RP}
As the perturbation parameter $\varepsilon \to 0$, the solution of the perturbed Riemann problem \eqref{eq: main_system}–\eqref{initial_data} converges to the solution of the corresponding Riemann problem \eqref{eq: main_system}–\eqref{Riemann_data} for all $(h, b)^\top \in \mathcal{U}$ in the space of Radon measures $\mathcal{M}_{Loc}(\mathbb{R};\mathbb{R}^2)$. Furthermore, as $t\rightarrow \infty$, the solution of the perturbed Riemann problem is governed by the initial states $(h_-, b_-)$ and $(h_+, b_+)$, which shows that the solution of the Riemann problem \eqref{eq: main_system}–\eqref{Riemann_data} is stable under small local perturbations of the initial data.
\end{theorem}
\begin{remark}
Theorem \ref{vanishing_theorem} proves that the Riemann solutions for \eqref{eq: main_system} remain stable under the gravity and surface tension parameters. This also indicates that the nonlinear flux due to gravity is a stable flux perturbation for the thin film system driven by only surface tension and vice versa.
\end{remark}

The remainder of the article is arranged as follows. In Section \ref{sec: modelling}, we develop the mathematical model \eqref{eq: main_system}. In Section \ref{sec: 2}, we discuss the hyperbolicity of the system \eqref{eq: main_system} and develop a class of strictly convex entropies. Section \ref{Sec: Riemann} discusses the Riemann problem \eqref{eq: main_system}-\eqref{Riemann_data} and analyze the behaviour of the system with vanishing physical parameters. In Section \ref{sec: 3}, we develop the solutions of the Cauchy problem \eqref{eq: main_system}-\eqref{initial_data} and prove that the Riemann solutions remain stable as $\varepsilon \rightarrow 0$. In Section \ref{sec: numerics}, we validate our analysis via some numerical examples. Concluding remarks and future outlooks are provided in Section \ref{sec: 5}.

\section{A hyperbolic model for the thin film flow under the influence of gravity and solute transport}\label{sec: modelling}
In this section, we develop the hyperbolic model \eqref{eq: main_system}, which governs the first-order dynamics of a thin film flow under the influence of a perfectly soluble anti-surfactant solute and gravity. As the model considered here is a natural extension of the model developed in \cite{barthwal2023construction, conn2016fluid}, we leave most of the physical details of the model and refer the interested reader to \cite{barthwal2023construction, conn2016fluid} for more details; see also \cite{barthwal2025hyperbolic, barthwal2025generalized}.
\begin{figure}
    \centering
    \includegraphics[width=0.75\linewidth]{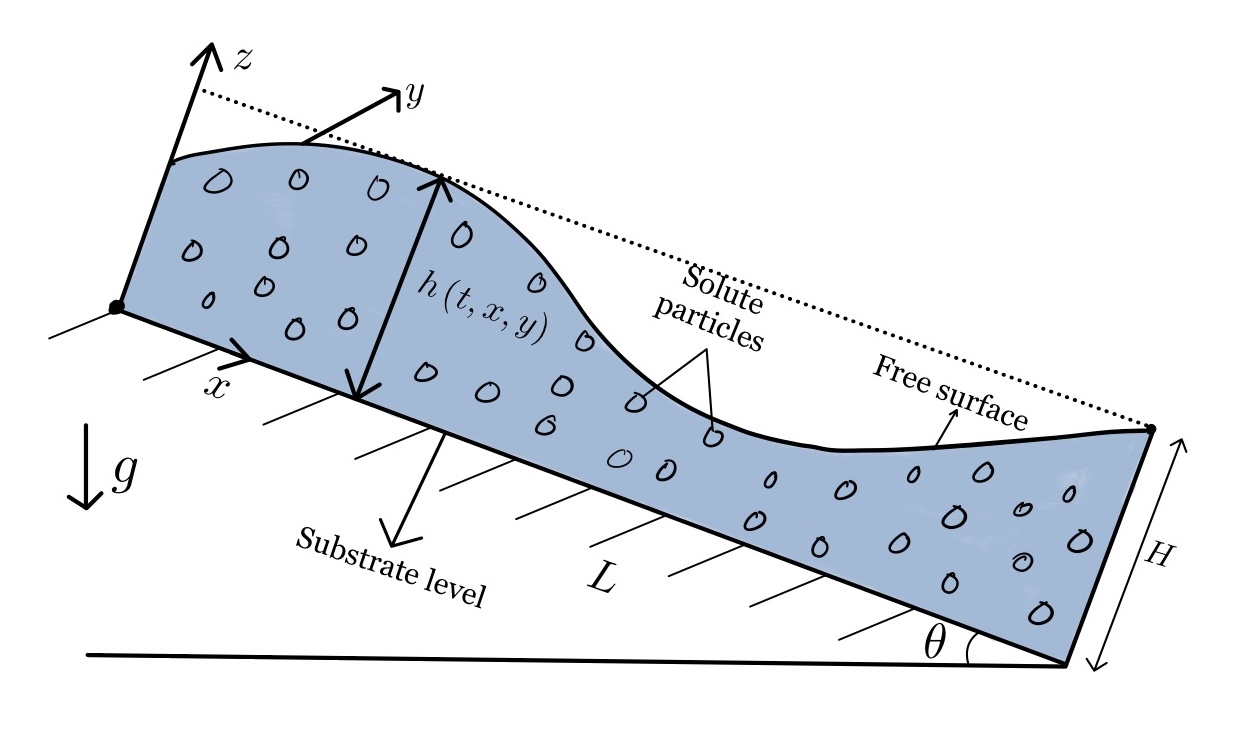}
    \caption{A thin film flow over an inclined plane with solute transport.}
    \label{fig:thin_film}
\end{figure}
\subsection{Evolution equations of film height and concentration}
Here we consider the model considered in \cite{barthwal2023construction} and \cite{conn2016fluid}, but consider the geometry such that the substrate is on an inclined plane. In this case, gravity effects can not be neglected. The free surface is assumed to be located at a height $z=h(x, y, t)$, see Figure \ref{fig:thin_film}. Following \cite{barthwal2023construction, conn2016fluid}, we assume that $\epsilon=H/L<\!<\! 1$ for thin film flow, where $H$ and $L$ are characteristic height and length, respectively. We use a similar non-dimensional scaling as in \cite{barthwal2023construction} and \cite{matar2004rupture}, and neglect the terms of order $\epsilon^2$ in the incompressible Navier-Stokes equations, along with the concentration transport equation for a perfectly soluble solute. This then leads to a reduced system of governing equations of the form
\begin{subequations}\label{eq: eq_of_motion}
\begin{align}
   \partial_x u+\partial_y v+\partial_z w&=0, \label{eq: 2.3a} \\
   \partial_x p&=\partial^2_z u+G\sin \theta,\label{eq: 2.3b}\\
   \partial_y p&=\partial^2_z v,\label{eq: 2.3b}\\
\partial_z p&=0, \label{eq: 2.3d}\\
    \partial_t c_0+\mathbf{u}\cdot\nabla c_0&=(\mathrm{P_{b}})^{-1}\left(\Delta c_0+\partial^2_z c_1\right).\label{eq: 2.3e} 
\end{align}
\end{subequations}
In \eqref{eq: 2.3a}-\eqref{eq: 2.3e}, $\mathbf{u}=(u, v)$ is the $x-y$ components while $w$ is the $z$-component of the non-dimensional velocity vector, $p$ is the pressure, $\theta$ is the inclination angle, $G$ is the non-dimensional gravity coefficient, $\Delta=(\partial_x^2, \partial_y^2)$ is the Laplacian operator, $\mathrm{P_b}$ is the Peclet number with $\mathrm{Pb}=O(1)$ and the concentration $c$ is assumed to satisfy an expansion of the form \cite{barthwal2023construction}
\[c=c_0(x, y, t)+\epsilon^2 c_1(x, y, z, t)+O(\epsilon^4).\]

Moreover, by neglecting the terms of order $\epsilon^2$ and simplifying, we obtain non-dimensional boundary conditions of the form \cite{barthwal2023construction}
\begin{subequations}\label{eq:2.17main}
\begin{align}
u=v=w   & = 0                                     & \mathrm{on}\hspace{0.5 cm}\{z=0\}, \label{eq:2.4a}\\
\partial_t h+\mathbf{u}\cdot \nabla h 
      & = w                                     & \mathrm{on}\hspace{0.5 cm}\{z=h\}, \label{eq:2.4c}\\
-p    & = (\mathrm{Ca})^{-1}\Delta h          & \mathrm{on}\hspace{0.5 cm}\{z=h\}, \label{eq:2.4f}\\
\partial_z \mathbf{u} 
      & = \mathrm{Ma}\nabla c_0                 & \mathrm{on}\hspace{0.5 cm}\{z=h\}, \label{eq:2.4g}\\
\partial_z c_1 
      & = 0                                     & \mathrm{on}\hspace{0.5 cm}\{z=0\}, \label{eq:2.4i}\\
\partial_z c_1 
      & = \nabla c_0\cdot \nabla h              & \mathrm{on}\hspace{0.5 cm}\{z=h\}. \label{eq:2.4j}
\end{align}
\end{subequations}
In \eqref{eq:2.4a}-\eqref{eq:2.4j}, $\mathrm{Ca}$ is the Capillarity number, $\mathrm{Ma}$ is the Marangoni number both of first order and $\nabla=(\partial_x, \partial_y)$ is the usual gradient operator (see \cite{conn2016fluid} for more details on the scalings of these dimensionless numbers).
\subsubsection{Evolution equation for the film thickness.}
Using the kinematic boundary conditions \eqref{eq:2.4c}, the evolution equation of film thickness is given by
\begin{align}\label{film_height}
\partial_t h + \int_{0}^{h} u(z)\,dz + \int_{0}^{h} v(z)\,dz = 0,
\end{align}
where the velocity components are obtained using \eqref{eq: eq_of_motion}-\eqref{eq:2.17main} and are given by
\begin{align*}%\label{eq: 2.22mainb}
   u(z)&=\left(\partial_x p-G\sin \theta\right)z(z/2-h)+z\partial_x c_0,\\  
   v(z)&= z(z/2-h)\partial_y p.
\end{align*}
Therefore, the evolution equation for the film thickness can now be obtained from \eqref{film_height}, and reads as
\begin{align}\label{eq: film_height}
    \partial_t h+\nabla.\left(\dfrac{h^3}{3{\rm{Ca}}}\nabla \Delta h+\dfrac{h^2 {\rm{Ma}}}{2}\mathbf{b}\right)+\partial_x \left(\dfrac{h^3 G\sin \theta}{3}\right)=0,
\end{align}
where $\mathbf{b}=\left(\partial_x c_0, \partial_y c_0\right)^\top$ is the concentration gradient vector.

\begin{comment}
Now we have
\begin{align*}
p=-(\mathrm{Ca})^{-1}\Delta^2 h
\end{align*}
Hence, by \eqref{eq: 2.3b}, we have
\begin{align*}
    \partial_z u=\bigg[\partial_x p-G\sin \theta\bigg]z+a_1. 
\end{align*}
Now from the tangential stress balance \eqref{eq: 2.4g}, we obtain
\begin{align*}
    a_1= (G\sin \theta-\partial_x p)h+\partial_x c_0.
    \end{align*}
Using the no-slip condition \eqref{eq: 2.4a}, we get
\end{comment}
\subsubsection{Evolution equation for the bulk concentration.}
The evolution equation for the leading order term $c_0$ can be obtained by integrating \eqref{eq: 2.3e} with respect to $z$ from $\{z=0\}$ to $\{z=h\}$ with boundary conditions \eqref{eq:2.4i} and \eqref{eq:2.4j}. The evolution equation of $c_0(x, y,  t)$ reads as
\begin{equation}\label{eq: concentration}
    h\partial_t c_0+Q_1 \partial_x c_0+Q_2\partial_y c_0-(\mathrm{P_{b}})^{-1}\nabla \cdot(h\mathbf{b})=0,
\end{equation}
where $Q_i$, $i\in \{1, 2\}$ are the volumetric fluxes defined as
\begin{align*}
    Q_1&=(G\sin \theta-\partial_x p)\dfrac{h^3}{3}+\dfrac{h^2}{2}\partial_x c_0,\\
    Q_2&=(-\partial_y p)\dfrac{h^3}{3}+\dfrac{h^2}{2}\partial_y c_0.
\end{align*}.
\subsection{Reduced evolution equations for first-order dynamics}
Following \cite{conn2017simple}, we assume that the capillarity and diffusivity effects are negligible and choose the velocity and length scales such that $\mathrm{Ma}=2\alpha$ and $G\sin \theta=\kappa$, where $\alpha$ and $\kappa$ are nonnegative scaled dimensionless parameters. Under these simplified assumptions, we can leave the higher-order terms from the evolution equations \eqref{eq: film_height}-\eqref{eq: concentration} and obtain reduced first-order governing equations of the form
\begin{equation}\label{eq:2.9new}
\begin{aligned}
 \frac{\partial h}{\partial t}
 + \nabla\cdot\!\left(\alpha h^{2}\mathbf{b}\right)
 + \frac{\partial}{\partial x}\!\left(\frac{\kappa h^{3}}{3}\right) &= 0,\\[0.3em]
 \frac{\partial c}{\partial t}
 + \alpha h\lvert\mathbf{b}\rvert^{2}
 + \frac{\kappa h^{2}}{3}\left(\frac{\partial c}{\partial x}\right) &= 0.
\end{aligned}
\end{equation}
where we have dropped the subscript from $c_0$.

Now, we differentiate the second equation in \eqref{eq:2.9new} with respect to $x$ and $y$ and  obtain
for the film thicknesses $h$ and the concentration gradient $\mathbf{b}=\left(\partial_x c, \partial_y c\right)^\top=(b_1, b_2)^\top$
the following set of nonlinear, first-order equations: 
\begin{equation}\label{eq: Main_system_new_updated}
\begin{aligned}
    \dfrac{\partial h}{\partial t}+\alpha \nabla\cdot\left(h^2\mathbf{b}\right)+\dfrac{\partial}{\partial x}\left(\dfrac{\kappa h^3}{3}\right)&=0,\vspace{0.2 cm} \\
    \dfrac{\partial \mathbf{b}}{\partial t}+\alpha \nabla\cdot\left(h|\mathbf{b}|^2 \mathbf{I}\right)+\nabla\left(\dfrac{\kappa h^2b_1}{3}\right)&=0.
 \end{aligned}
\end{equation}
In particular, for the one-dimensional case, i.e., when the $y$ component is zero, we obtain the governing system \eqref{eq: main_system}.
\section{Hyperbolicity of the system \eqref{eq: main_system} and the local wellposedness of Cauchy problem}\label{sec: 2}
In this section, we discuss the basic properties of the system \eqref{eq: main_system}, which is the one-dimensional version of \eqref{eq: Main_system_new_updated} and develop a class of strictly convex entropies for it.
\subsection{Hyperbolicity and Riemann invariants of the system \eqref{eq: main_system}}
For smooth solutions, \eqref{eq: main_system} can be written in the primitive form
\begin{align}
\mathbf{U}_t+\mathbf{DF(U)}~\mathbf{U}_x=0,
\end{align}
where \[\mathbf{U}=\begin{bmatrix}h\\ b
\end{bmatrix}, ~\text{and}~\mathbf{DF(U)}=\begin{bmatrix}
    \alpha hb+\kappa h^2 & \alpha h^2/2\\
    \alpha b^2/2+2\kappa hb/3& \alpha hb+\kappa h^2/3
\end{bmatrix}.\]
By a direct calculation, the eigensystem of the system \eqref{eq: main_system} is given by
\begin{align}\label{eigensystem}
\lambda_1(h, b) &= \alpha hb+\dfrac{\kappa h^2}{3}, \quad &\mathbf{r}_1&=(-3\alpha h, 3\alpha b+2\kappa h)^\top , \\
\lambda_2(h, b) &= 3\alpha hb+\kappa h^2, \quad&\mathbf{r}_2& = (h, b)^\top. 
\end{align}
A straightforward calculation leads us to $\nabla_\mathbf{U} \lambda_1\cdot \mathbf{r}_1 = 0$, while 
\[\nabla_\mathbf{U} \lambda_2\cdot \mathbf{r}_2 = h(3\alpha b+2\kappa h)\neq 0~\mathrm{for}~(h, b)^\top\in \mathcal{U}^{\circ}.\]
This implies that the associated characteristic field of $\lambda_1$ is linearly degenerate while that of $\lambda_2$ is either genuinely nonlinear in the interior or linearly degenerate on the borderlines of the state space $\mathcal{U}$. Moreover, we have a clear ordering of the eigenvalues given by 
\[\lambda_1(h, b) < \lambda_2(h, b)=3\lambda_1(h, b),~{\rm{for}}~(h, b)^\top\in \mathcal{U}^{\circ}\] 
and 
\[\lambda_1(h, b) = \lambda_2(h, b) = 0,~{\rm{for}}~ (h, b)^\top\in \partial \mathcal{U}.\] 

This implies that the system \eqref{eq: main_system} is strictly hyperbolic in the interior of the state space \eqref{state_space} while for $h=0$ (at the boundaries), the system \eqref{eq: main_system} becomes weakly hyperbolic.

From the characteristic analysis, it follows that the first characteristic field is always linearly degenerate. Consequently, the associated wave is a contact discontinuity $J$. On the other hand, the second characteristic field is genuinely nonlinear. Therefore, the corresponding wave is either a rarefaction wave $R$ or a shock wave $S$ whenever the state variables $(h,b)^\top \in  \mathcal{U}^{\circ}$. For states on the boundary of the admissible set, i.e., when $(h,b)^\top \in \partial \mathcal{U}$, the wave structure is no longer purely classical. In this case, one may encounter measure-valued solutions or composite waves. We discuss these situations in Section \ref{sec: Riemann_solution_at_boundary}.
\begin{remark}
For the case when $\alpha=0$, $b=0$ and $\kappa=1$, the system \eqref{eq: main_system} collapses to the thin film equation under the influence of gravity, which has been analyzed as a prototype example of scalar conservation laws with nonconvex flux for which the existence of nonclassical waves such as undercompressive shock waves is a well-known result, see e.g., \cite{bertozzi1999undercompressive, Chalonsetal} and references cited therein. Therefore, it would be interesting to analyze whether the formal limit $b\rightarrow 0$ of the system \eqref{eq: main_system} leads to undercompressive shock waves. However, we don't discuss that situation in this article.     
\end{remark}
\subsubsection{Riemann invariants of the system \eqref{eq: main_system}}
By definition, the Riemann invariants $w_i(h, b)$ are parallel to the left eigenvectors such that for $i\in \{1, 2\}$
\[\nabla_{\mathbf{U}} w_i\cdot r_i=0,\]
where $\nabla_{\mathbf{U}}=(\partial_h, \partial_b)^\top$. Therefore, one can easily obtain the Riemann invariants given by
\begin{align}\label{Riemann_invariants}
      w_1= \alpha hb+\dfrac{\kappa h^2}{3}, \qquad w_2=\dfrac{b}{h}.
\end{align}
In view of explicit Riemann invariants, we can obtain a class of strictly convex entropies for the system \eqref{eq: main_system} defined in the regions where the system remains strictly hyperbolic.
\subsection{Entropy-entropy pairs for the system \eqref{eq: main_system} in $\mathcal{U}^{\circ}$}
In this section, we obtain a class of entropy-entropy pairs $(\eta, q)$ for the system \eqref{eq: main_system}. Moreover, we provide the sufficient conditions under which the obtained entropies are strictly convex. Precisely, we prove the following.
\begin{theorem}[Strictly convex entropies]\label{entropy_theorem}
Let $\Psi$ and $\Theta$ be arbitrary functions in  $C^1(0,\infty)$. There exists a class of entropy/entropy-flux pairs $(\eta, q)$ for the system \eqref{eq: main_system} defined in $\mathcal{U}^{\circ}$. They are given by
\begin{equation}\label{eq:entropyclass}
\begin{array}{rcl}
 \eta[h, b]&=& \Psi(w_1)+\sqrt{w_1}\Theta\left(p\right),\\[1ex]
 q[h, b]&=&3\left(w_1\Psi(w_1)-\displaystyle\int^{w_1} \Psi(w_1)\, dw_1\right)+(w_1)^{3/2}\Theta(p),
\end{array}
\end{equation}
where $w_1$ and $w_2$ are the Riemann invariants of the system \eqref{eq: main_system} defined in \eqref{Riemann_invariants} and $p$ is a function of $w_2$ defined by $p=3\alpha w_2+\kappa$.  

Moreover, a sufficient condition for the entropy $\eta[h, b]$ to be strictly convex is that the functions $\Psi$ and $\Theta$ satisfy
\begin{align}
\Psi''(w_1)>0, ~\Psi'(w_1)<0,~2w_1\Psi''(w_1)+\Psi'(w_1)>0, ~\Theta(p)= A \sqrt{p}+\dfrac{B}{\sqrt{p}},   
\end{align}
for some arbitrary nonnegative constants $A$ and $B$.
\end{theorem}
\begin{proof}
Since $w_1=\alpha hb+\dfrac{\kappa h^2}{3}$ and $w_2=3\alpha hb+\kappa h^2$ are the Riemann invariants of the system \eqref{eq: main_system}, we must have
\begin{equation}\label{diagonal_system}
\begin{aligned}
    (w_1)_t+3w_1 (w_1)_x=0,\\
    (w_2)_t+w_1 (w_2)_x=0,
    \end{aligned}
\end{equation}
where we have used the fact that $\lambda_2(w_1, w_2)=3\lambda_1(w_1, w_2)=3w_1$.

Thus by comparing the mixed derivatives of $q$, any entropy pair $(\eta, q)[w_1, w_2]$ must satisfy
\begin{align*}
    2w_1\eta_{w_1w_2}-\eta_{w_2}=0,
\end{align*}
which implies that any entropy $\eta$ should be of the form 
\begin{align}
    \eta[w_1, w_2]=\Psi(w_1)+\sqrt{w_1}\Theta(w_2),
\end{align}
along with the entropy flux given by
\begin{align}
    q[w_1, w_2]=3\left(w_1\Psi(w_1)-\displaystyle\int \Psi(w_1)\, dw_1\right)+(w_1)^{3/2}\Theta(w_2).
\end{align}
In view of the diagonal system \eqref{diagonal_system}, a necessary and sufficient condition \cite{dafermos2005hyperbolic} for $\eta$ to be strictly convex is that 
\[\mathbf{r_i^\top} \nabla^2_{(h, b)} \eta(w_1, w_2) \mathbf{r_i}>0 ~\text{for}~ i\in \{1, 2\},\]
where $\mathbf{r}_i$ are the eigenvectors of the system defined in \eqref{eigensystem} and $\nabla^2_{(h, b)}$ is the Hessian operator w.r.t. $(h, b)$. 

First, we use the Hessian identity
\begin{align}\label{hessian_identity}
 \nabla^2_{(h, b)} \eta(w_1, w_2)=D^\top \nabla^2_{(w_1, w_2)} \eta(w_1, w_2) D+\eta_{w_1}\nabla^2_{(h, b)} w_1+\eta_{w_2}  \nabla^2_{(h, b)} w_2, 
\end{align}
where $D$ denotes the Jacobian of $(w_1, w_2)$ w.r.t. $(h, b)$.

An easy computation gives 
\[D\mathbf{r_1}=(2w_1, 0)^\top ~\text{and}~D\mathbf{r_2}=(0, 6\alpha (3\alpha w_2+\kappa))^\top=(0, 6\alpha p)^\top,\]
where $p=3\alpha w_2+\kappa$. Therefore, for $i\in \{1, 2\}$ we multiply \eqref{hessian_identity} by $\mathbf{r_i^\top}$ from the left and by $\mathbf{r_i}$ from the right, and obtain 
\begin{equation}\label{hessian_identity_new}
\begin{aligned}
 \mathbf{r_1^\top} \nabla^2_{(h, b)} \eta(w_1, w_2)\mathbf{r_1}=4w_1^2\eta_{w_{1}w_1}+\eta_{w_1}\mathbf{r_1^\top} \nabla^2_{(h, b)} w_1\mathbf{r_1}+\eta_{w_2}  \mathbf{r_1^\top}\nabla^2_{(h, b)} w_2\mathbf{r_1}, \vspace{0.2 cm}\\
 \mathbf{r_2^\top} \nabla^2_{(h, b)} \eta(w_1, w_2)\mathbf{r_2}=36\alpha p^2\eta_{w_{2}w_2}+\eta_{w_1}\mathbf{r_2^\top}\nabla^2_{(h, b)} w_1\mathbf{r_2}+\eta_{w_2}  \mathbf{r_2^\top}\nabla^2_{(h, b)} w_2\mathbf{r_2}.
 \end{aligned}
\end{equation}
A tedious but straightforward calculation then leads to 
\begin{align}\label{hessian_identity_new}
 \mathbf{r_1^\top}\nabla^2_{(h, b)} \eta(w_1, w_2)\mathbf{r_1}&=2w_1(2w_1\Psi''(w_1)+\Psi'(w_1)), \vspace{0.2 cm}\\
 \mathbf{r_2^\top} \nabla^2_{(h, b)} \eta(w_1, w_2)\mathbf{r_2}&=9\alpha^2\sqrt{w_1}\left(4p^2\Theta''(p)+4p\Theta'(p)-\Theta(p)-2w_1\Psi'(w_1)\right).
\end{align}
Clearly, in view of positivity of $w_1$ and $w_2$ in $\mathcal{U}^{\circ}$, a sufficient condition for the positivity of $\mathbf{r_1^\top} \nabla^2_{(h, b)} \eta(w_1, w_2)\mathbf{r_1}$ and $\mathbf{r_2^\top}\nabla^2_{(h, b)} \eta(w_1, w_2)\mathbf{r_2}$ is 
\[4p^2\Theta''(p)+4p\Theta'(p)-\Theta(p)=0,~ \Psi'(w_1)<0,  ~\Psi''(w_1)>0, ~2w_1\Psi''(w_1)+\Psi'(w_1)>0.\]
This concludes the proof of the theorem.
\end{proof}
\begin{remark}
The simplest choice for the functions $\Psi, \Theta$ in Theorem \eqref{entropy_theorem} is $\Psi(w_1)=\dfrac{1}{3w_1}$ and $\Theta(p)={\sqrt{\dfrac{3}{p}}}$. This gives the strictly convex entropy as 
\[\eta(h, b)=h+\dfrac{1}{3\alpha hb+\kappa h^2}.\]
Actually, any function of the form $\Psi(w_1)=(w_1)^{-n}, n>0$ leads to a strictly convex entropy.
\end{remark}
\begin{proof}[Proof of Theorem \ref{local_wellposedness}]
In view of the strictly convex entropies, system \eqref{eq: main_system} is a Friedrichs symmetrizable system and thus local well-posedness of smooth solutions in $\mathcal{U}^{\circ}$ can be ensured using the classical theory of hyperbolic conservation laws (see e.g. \cite{dafermos2005hyperbolic, MR390516, majda2012compressible, serre1999systems} and references cited therein). Theorem \ref{local_wellposedness} is proved. 
\end{proof}
\section{The Riemann problem}\label{Sec: Riemann}
In this section, we first compute the elementary wave curves for the system \eqref{eq: main_system}. Based on these wave curves and a clear entropy structure, we find explicit intermediate states for the Riemann problem \eqref{eq: main_system}-\eqref{Riemann_data}.
\subsection{Elementary classical waves of system \eqref{eq: main_system} for $(h, b)^\top\in \mathcal{U}^{\circ}$}
For a given left state $(h_l, b_l)$ in the interior of the state space \eqref{state_space}, through a tedious but routine calculation, one can easily find the state set of $(h, b)$ which can be linked to $(h_l, b_l)$ by a contact
discontinuity $J(h_l, b_l)$, a rarefaction wave $R(h_l, b_l)$, or a shock wave $S(h_l, b_l)$. We describe these waves explicitly in this section.
\begin{comment}
\subsubsection{Equilibrium points and smooth travelling waves}
We first look for smooth travelling waves of the system \eqref{eq: main_system}, where the two end states are the equilibrium (steady state) points of the system \eqref{eq: main_system}. We assume the solution ansatz of the form $(h, b)(\xi)=(h, b)(x-ct), ~c\in \mathbb{R}$ and from \eqref{eq: main_system}, we obtain a system of ODEs
\begin{align*}
    -cU'+(U(\alpha hb+\dfrac{\kappa h^2}{3}))'=0,\\
    -cV'+(V(\alpha hb+\dfrac{\kappa h^2}{3}))'=0.
\end{align*}
On integrating the two equations, we obtain
\begin{align*}
 (U(\alpha hb+\dfrac{\kappa h^2}{3}-c))=A,\\
(V(\alpha hb+\dfrac{\kappa h^2}{3}-c))=B,
\end{align*}
where $A$ and $B$ are the end states of $U$ and $V$, respectively as $\xi\rightarrow \pm \infty$. Assume that the two end states around $\xi=0$ are nonzero equilibrium states such that there is no jump in the two states, i.e. $(h, b)\rightarrow (U_*, V_*)$ as $\xi\rightarrow \pm \infty$. For these constant equilibrium states, we must have $A=B=0$ and thus all the nontrivial equilibrium states must lie on the center manifold $\phi_1(UV)+\phi_2(U)=c$, which corresponds to the level sets of the first Riemann invariant. For all non-trivial constant states lying on these level sets, the existence of smooth travelling waves can be ensured due to \cite{} as the system becomes linearly degenerate along these level sets.
\end{comment}
\subsubsection{Rarefaction waves.}
Since the $2^\mathrm{nd}$ characteristic field is genuinely nonlinear, the wave associated with this field is either a rarefaction or a shock wave. Moreover, the 2-Riemann invariant remains constant across the fans of 2-rarefaction wave. This then leads to the expression of the 2-rarefaction wave given by
\begin{align}\label{eq: 4.2R} R_2:=
    \begin{cases}
        \dfrac{dx}{dt}=\lambda_2=3\alpha hb+\kappa h^2,\vspace{0.2 cm}\\
            \dfrac{b}{h}=\dfrac{b_l}{h_l}, \vspace{0.2 cm}\\
            h_l\leq h,~ b_l\leq b.
    \end{cases}
\end{align}
Clearly, $R_2$ is a curve in  the $(h, b)$-space emanating from $\mathbf{U}_l=(h_l, b_l)^\top$.
\subsubsection{Discontinuous solutions: shock and contact waves.}
For a given left state $\mathbf{U}_l=(h_l, b_l)^\top$, the discontinuous solutions (shock or contact discontinuity) with speed $\sigma\in \mathbb{R}$ are defined by
\begin{align} \label{wave}
    \mathbf{U} (x, t)=\begin{cases}
        \mathbf{U}_l=(h_l, b_l)^\top\in {\mathcal U}&:~x<\sigma t,\\
        \mathbf{U}_r=(h_r, b_r)^\top \in {\mathcal U}&:~x>\sigma t.
    \end{cases}
\end{align}
Moreover, \eqref{wave} is a weak solution of \eqref{eq: main_system} provided the Rankine-Hugoniot (R-H) conditions 
\begin{align}\label{RH}
\sigma[\![\mathbf{U}]\!]  =[\![\mathbf{F(U)}]\!]
\end{align}
hold. In \eqref{RH},  $[\![\mathbf{V}]\!]=\mathbf{V}_r-\mathbf{V}_l$ denotes the jump in $\mathbf{V}$ for $\mathbf{V}_l, \mathbf{V}_r\in \mathbb{R}^2$.

In view of the system \eqref{eq: main_system}, the R-H relations then are
\begin{equation}\label{eq: 4.8}
    \begin{array}{rclrcl}
    \sigma[\![h]\!]\!\!\!\!&=&\!\!\!\!\Big[\!\!\Big[h\left(\alpha hb+\dfrac{\kappa h^2}{3}\right)\Big]\!\!\Big],\vspace{0.2 cm}\\
    \sigma[\![b]\!]\!\!\!\!&=&\!\!\!\!\Big[\!\!\Big[b\left(\alpha hb+\dfrac{\kappa h^2}{3}\right)\Big]\!\!\Big].
    \end{array}
\end{equation}
For nonconstant solutions and a given left state $\mathbf{U_l}=(h_l, b_l)^{\top}\in   \mathcal{U}^{\circ}$, the discontinuity curves associated with the first characteristic field are contact discontinuity curves through $\mathbf{U_l}$ in the $(h, b)$-space, which are given by
\begin{align}\label{contact_relations}
    J_1:=\begin{cases}
    \mu_1((h_l, b_l);(h, b))=h\left(\alpha b+\dfrac{\kappa h}{3}\right) = h_l\left(\alpha b_l+\dfrac{\kappa h_l}{3}\right),\\
    h\neq h_l, \quad b\neq b_l.
    \end{cases}
\end{align}
The corresponding $1$-contact wave \eqref{wave} is an entropy solution of \eqref{eq: main_system} by definition.

The $2$nd characteristic field is genuinely nonlinear and thus the associated discontinuous solution \eqref{wave} is a 2-shock wave, which satisfies the R-H condition \eqref{RH} and is entropy admissible. For the left state $\mathbf{U_l}$, 2-shock can be expressed as 
\begin{align}\label{shock_relations}
   S_2:=\begin{cases} \sigma_2((h_l, b_l);(h, b))= (h+h_l)\left(\alpha b+\dfrac{\kappa h}{3}\right)+h_l\left(\alpha b_l+\dfrac{\kappa h_l}{3}\right),\\ 
   \dfrac{b}{h}=\dfrac{b_l}{h_l},
   \\
   h<h_l,\quad b<b_l.\end{cases}
\end{align}
Since the system \eqref{eq: main_system} is strictly hyperbolic in $\mathcal{U}^{\circ}$, the entropy admissibility of the shock $S_2$ with respect to the entropy/entropy-flux pair is equivalent to the Lax entropy conditions. For the $2$-shock  wave with the right state $\mathbf{U}$, Lax entropy conditions are given by
\begin{align}\label{eq: 25}
    \lambda_2(\mathbf{U})<\sigma_2<\lambda_2(\mathbf{U_l}),\,
 \lambda_1(\mathbf{U_l})  < \sigma_2.
\end{align}
This results in the last inequality in \eqref{shock_relations}.  
\begin{remark}
Note that the shock curve and rarefaction curve associated with the $2^{nd}$-characteristic field are the same and form a straight line in the phase space. This shows that the Riemann invariant $w_2=b/h$ is associated with a Temple field \cite{temple1983systems}. However, it is easy to show that the Riemann invariant $w_1=\alpha hb+\kappa h^2/3$ is not associated with a Temple field as it violates the condition of Lemma 4.2 in \cite{MR892021}, if we use a viscous approximation of \eqref{eq: main_system}. This indicates that the system \eqref{eq: main_system} is not a Temple system completely. However, we can still show that the solution of the Riemann problem can be developed for any initial Riemann data \eqref{Riemann_data} satisfying the assumption (\assref). 
\end{remark}
\begin{remark}
For the case $h_r\rightarrow 0$, i.e., when the right state lies on the boundary of the state space, the speed of shock and contact discontinuity coincide and merge to become an overcompressive delta shock as we see in the next section.
\end{remark}

\subsection{Proof of Theorem \ref{theorem:Riemann}: Solution of the Riemann problem}\label{sec: Riemann_solutions_construction}
In this section, we explicitly calculate the intermediate states to prove that for all choices of initial Riemann data in $\mathcal{U}$, a unique solution of the Riemann problem \eqref{eq: main_system}-\eqref{Riemann_data} exists and remains in $\mathcal{U}$. Depending on initial data in $\mathcal{U}$, there are several possible situations for the solution of \eqref{eq: main_system}-\eqref{Riemann_data}, which are described below.
\subsubsection{Solution in the interior of the state space.}
We first develop the unique Riemann solutions for the case when the Riemann data \eqref{Riemann_data} lies in $\mathcal{U}^{\circ}$. Since the shock curve and rarefaction curve overlap with each other, the intermediate state between a contact discontinuity ($J$) and shock/rarefaction ($S/R$) can be obtained for any given Riemann data \eqref{Riemann_data} in a unique manner. Precisely, assume that the intermediate state between $J$ and $S/R$ is $(h_*, b_*)$, see \ref{fig: J+S}-\ref{fig: J+R}. Then for consistency, $(h_*, b_*)$ must satisfy
\begin{align*}
    \alpha h_-b_-+\dfrac{\kappa h_-^2}{3}=\alpha h_*b_*+\dfrac{\kappa h_*^2}{3},\quad \dfrac{b_*}{h_*}=\dfrac{b_+}{h_+}.
\end{align*}
Thus $(h_*, b_*)$ can be obtained uniquely in the interior of the state space and is given by
\begin{align}\label{intermediate}
 (h_*, b_*)=   \left(\sqrt{\dfrac{h_-h_+(3\alpha b_-+\kappa h_-)}{3\alpha b_++\kappa h_+}}, b_+\sqrt{\dfrac{h_- (3\alpha b_-+\kappa h_-)}{h_+(3\alpha b_++\kappa h_+)}}\right)\in \mathcal{U}.
\end{align}
\begin{figure}
    \centering
    \begin{subfigure}{0.48\linewidth}
\includegraphics[width=\linewidth]{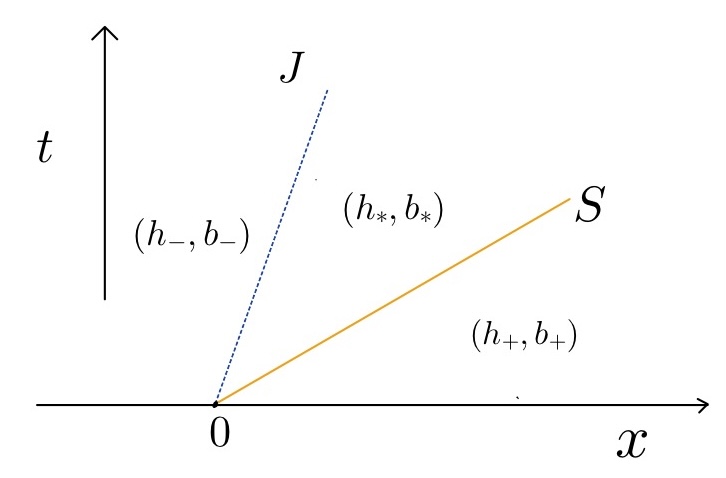}
    \caption{Solution of the Riemann problem for $0<h_-b_-<h_+b_+$ $(J+S)$.}
    \label{fig: J+S}
    \end{subfigure}\hfill
    \begin{subfigure}{0.48\linewidth}
        \includegraphics[width=\linewidth]{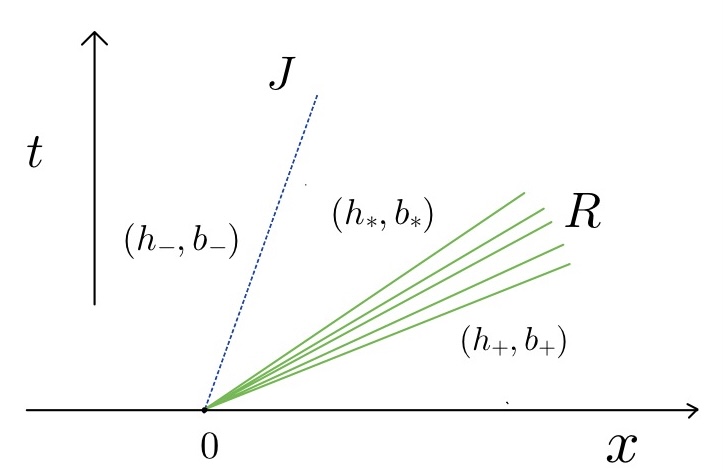}
    \caption{Solution of the Riemann problem for $0<h_+b_+<h_-b_-$ $(J+R)$.}
    \label{fig: J+R}
    \end{subfigure}
\end{figure}

In view of the explicit intermediate state, we can describe the Riemann solutions as follows:\vspace{0.2 cm}\\
\textbf{Case I: $h_\pm, b_\pm\in \mathcal{U}^{\circ}, ~0<h_-b_-<h_+b_+$.}\\
For this case, the Riemann solution consists of a contact discontinuity $J$ followed by a rarefaction wave $R$. Precisely, the solution structure is given by

%figure2

\begin{align}\label{J+R}
\left(h, b\right)(x, t)= \left\{
        \begin{array}{llll}
            (h_-, b_-),& -\infty<x<\mu_1((h_-, b_-); (h_*, b_*))t,\vspace{0.2 cm}\\
            \left(h_*, b_*\right),&\mu_1((h_-, b_-); (h_*, b_*))t<{x}\leq \lambda_2(h_*, b_*)t,\vspace{0.2 cm}\\
            (h, b)(x/t),&\lambda_2(h_*, b_*)t\leq x\leq \lambda_2(h_+, b_+)t,\vspace{0.2 cm}\\
            (h_+, b_+),&\lambda_2(h_+, b_+)t\leq x<+\infty,\\
        \end{array}
    \right.
\end{align}
where the state $(h, b)$ inside the rarefaction wave $R$ can be determined uniquely and is given by
\begin{align}\label{rarefaction}
    (h, b)(x/t)=\left(\sqrt{\dfrac{x h_+}{t(3\alpha b_++\kappa h_+)}},b_+\sqrt{\dfrac{x}{h_+t(3\alpha b_++\kappa h_+)}} \right).
\end{align}\\
\textbf{Case II: $h_\pm, b_\pm\in \mathcal{U}^{\circ}, ~0<h_+b_+<h_-b_-$.}\\
For this case, the Riemann solution consists of a contact discontinuity $J$ followed by a shock wave $S$. Therefore, the solution structure is given by
\begin{align}\label{J+S}
\left(h, b\right)(x, t)= \left\{
        \begin{array}{lll}
            (h_-, b_-),& -\infty<x<\mu_1((h_-, b_-); (h_*, b_*)){t},\\
            (h_*, b_*),&\mu_1((h_-, b_-); (h_*, b_*)){t}<x<\sigma_2((h_*, b_*); (h_+, b_+)){t},\\
            (h_+, b_+), & \sigma_2((h_*, b_*); (h_+, b_+)){t}<x<+\infty,
         \end{array}
    \right.
\end{align}
where the intermediate state $(h_*, b_*)$ is defined in \eqref{intermediate}.
\begin{comment}
This implies that the propagation speed of the shock connecting the states $(u_*, v_*)$ and $(h_+, b_+)$ is given by
\[\sigma((h_*, b_*); (h_+, b_+))=(h_++h_*)\left(\alpha b_++\dfrac{\kappa h_+}{3}\right)+h_*\left(\alpha b_*+\dfrac{\kappa h_*}{3}\right).\]
\end{comment}
\subsubsection{Solution of the Riemann problem at the boundary of state space \eqref{state_space}.}\label{sec: Riemann_solution_at_boundary} We turn to the special cases in which the Riemann data \eqref{Riemann_data} satisfy assumption (\assref). With this additional condition, one can construct unique solutions to the Riemann problem \eqref{eq: main_system}-\eqref{Riemann_data}, even when the system \eqref{eq: main_system} is no longer strictly hyperbolic.\newline\\

\textbf{Case III: $h_-\in \partial \mathcal{U}$, $h_+, b_\pm\in \mathcal{U}^{\circ}$.}\\
If $h_-=0$, then the left state is perturbed slightly to $(h_-, b_-)$ similar to \cite{liu1980vacuum}, where $h_-$ is a sufficiently small positive number, which obeys $0<h_-<h_+$ and $0<b_-<b_+$. In this case, the solution of the Riemann problem is $J+R$ given in \eqref{J+R}. Also, it is easy to see that
\begin{align}
     \underset{h_-\rightarrow 0} {\lim}\mu_1(h_-, b_-)=\underset{h_-\rightarrow 0} {\lim} \lambda_2\left(\sqrt{\dfrac{h_-h_+(3\alpha b_-+\kappa h_-)}{3\alpha b_++\kappa h_+}}, b_+\sqrt{\dfrac{h_- (3\alpha b_-+\kappa h_-)}{h_+(3\alpha b_++\kappa h_+)}}\right)=0.
\end{align}
This implies that for the limit $h_-\rightarrow 0$, both the wave $J$ and the wave back of $R$ align together to form a composite wave $JR$ at the line $x=0$ in the $(x, t)$-plane. The solution in this case is thus given by
\begin{align}\label{composite_wave}
    (h, b)(x, t)=\begin{cases}
        (0, b_-), & -\infty<x<0,\\
        (h, b)(x/t), & 0\leq x\leq \lambda_2(h_+, b_+)t,\\
        (h_+, b_+), & \lambda_2(h_+, b_+)t<x<+\infty,
    \end{cases}
\end{align}
where the state $(h, b)(x/t)$ in $R$ is defined in \eqref{rarefaction}.\newline\\
\textbf{Case IV: $h_+\in \partial \mathcal{U}$, $h_-, b_\pm\in \mathcal{U}^{\circ}$.}\\
The reverse case of $h_+=0$ leads to a more interesting phenomenon. In this case
\begin{align}
     \underset{h_+\rightarrow 0} {\lim}\mu_1(h_-, b_-)=\underset{h_+\rightarrow 0} {\lim}\mu_1(h_*, b_*)=\underset{h_+\rightarrow 0} {\lim} \sigma_2=\alpha h_* b_*+\dfrac{\kappa h_*^2}{3}.
\end{align}
This shows that the two waves tend to the line $x=\mu_1(h_-, b_-)t$ of the contact discontinuity $J$ in the $(x, t)$-plane. Also, the intermediate state in this case satisfies
\begin{align}
    \underset{h_+\rightarrow 0} {\lim}(h_*, b_*)= \underset{h_+\rightarrow 0} {\lim} \left(\sqrt{\dfrac{h_-h_+(3\alpha b_-+\kappa h_-)}{3\alpha b_++\kappa h_+}}, b_+\sqrt{\dfrac{h_- (3\alpha b_-+\kappa h_-)}{h_+(3\alpha b_++\kappa h_+)}}\right)=(0, +\infty).
\end{align}
This motivates us to define a weighted $\delta$-measure solution (i.e., $\delta$-shock) supported on a discontinuity curve $\Gamma=\{(x(s), t(s)): ~a<x<b\}$. Following \cite{chen2003formation, keyfitz1995spaces, sheng1999Riemann, tan1994delta}, we consider a piecewise smooth Dirac valued solution of \eqref{eq: main_system}-\eqref{Riemann_data} satisfying the generalized Rankine-Hugoniot jump conditions
\begin{align}\label{GRH}
   \left\{
        \begin{array}{lll}
            \dfrac{dx}{dt}&= \sigma_{\delta},\\
            \sigma_{\delta}[\![h]\!]&=\Big[\!\!\Big[h\left(\alpha hb+\dfrac{\kappa h^2}{3}\right)\Big]\!\!\Big],\\
            \dfrac{d\beta(t)}{dt}&=\bigg(\sigma_{\delta}[\![b]\!]-\Big[\!\!\Big[b\left(\alpha hb+\dfrac{\kappa h^2}{3}\right)\!\Big]\!\!\Big]\bigg),
            \end{array}
    \right.
\end{align}
where $\sigma_\delta$ and $\beta(t)$ denote the speed and the strength of Dirac-valued shock wave, respectively. The expression of speed and strength of the $\delta$-shock wave can be obtained easily from the generalized Rankine-Hugoniot conditions and is given by
\begin{align}\label{strength_delta}
\sigma_{\delta}=\alpha h_- b_-+\dfrac{\kappa h_-^2}{3}, \quad \beta(t)=b_+\left(\alpha h_- b_-+\dfrac{\kappa h_-^2}{3}\right)t, \quad t\geq 0.
\end{align}
Moreover, for uniqueness, the generalized entropy conditions or overcompressibility conditions need to be imposed. They are given by
\begin{align}\label{eq: GEC}
0=\lambda_1(h_+, b_+)<\sigma_{\delta}=\lambda_1(h_-, b_-)<\lambda_2(h_-, b_-)=3\alpha h_-b_-+\kappa h_-^2,
\end{align}
which implies that all the characteristics on both sides of the $\delta$-shock are not outgoing. It is easy to see that, for all $b_\pm \in \mathcal{U}$, the generalized entropy conditions imply $0=h_+<h_-$. 

To summarize this case, we define the delta shock wave solution in the following lemma.
\begin{lemma}
For $h_-, b_\pm \in \mathcal{U}^{\circ}$ and $h_+\in \partial \mathcal{U}$, the solution of the Riemann problem \eqref{eq: main_system}-\eqref{Riemann_data} is a $\delta$-shock, which can be expressed as follows
\begin{align}\label{pre5}
\left(h, b\right)(x, t)= \left\{
        \begin{array}{lll}
            (h_-, b_-),& x<\sigma_{\delta}{t},\\
            (0,\beta(t)\delta(x-\sigma_{\delta}{t})),& x=\sigma_{\delta}{t},\\
            (0, b_+),& x>\sigma_{\delta}{t},\\
         \end{array}
    \right.
\end{align}
where $\sigma_\delta$ and $\beta(t)$ are defined in \eqref{strength_delta}.

Moreover, for any given test function $\varphi\in C_0^\infty(\mathbb{R}\times \mathbb{R}_+)$, the weighted $\delta$-measure $\beta(s)\delta_{\Gamma}$ supported on a smooth curve $\Gamma$ is defined by
\begin{align}
    \langle \beta(\cdot)\delta_{\Gamma}, \varphi(\cdot, \cdot)\rangle=\displaystyle \int_{a}^{b} \beta(s)\varphi(x(s), t(s))\sqrt{x'(s)^2+t'(s)^2}~ ds.
\end{align}
\end{lemma}
\begin{proof}
In order to prove this theorem, it is enough to prove that the function \eqref{pre5} satisfies
\begin{align}\label{eq: delta}
    \langle b, \varphi_t\rangle+\left\langle b\left(\alpha hb+\dfrac{\kappa h^2}{3}\right), \varphi_x \right\rangle=0
\end{align}
for any $\varphi\in C_0^\infty (\mathbb{R}\times \mathbb{R}_+)$, where
\begin{align*}
\langle b, \varphi_t\rangle&=\displaystyle\int_{\mathbb{R}_+}\int_{\mathbb{R}} b_0 \varphi dx dt+\langle\beta(t)\delta_{\Gamma}, \varphi\rangle,\\
\left\langle b\left(\alpha hb+\dfrac{\kappa h^2}{3}\right), \varphi_x\right\rangle&=\displaystyle\int_{\mathbb{R}_+}\int_{\mathbb{R}}b_0\left(\alpha h_0b_0+\dfrac{\kappa h_0^2}{3}\right)\varphi~ dx~ dt+\langle \sigma_{\delta}\beta(t)\delta_{\Gamma}, \varphi\rangle.
\end{align*}

The proof of equality \eqref{eq: delta} is straightforward and follows the same arguments as in \cite{sen2020delta} and \cite{shen2018delta}. For this reason, we donot present it here.
\end{proof}
\begin{remark}
If either $b_-$ or $b_+$ vanishes, the solution structure remains the same as in the cases in the interior of the state space, and therefore, no separate discussion of these cases is required.   
\end{remark}
\begin{remark}
The inverse relation between $h$ and $b$ in \eqref{intermediate} indicates that the rupture of film $h_*\rightarrow 0$ is possible in the Riemann solutions when $b_+\rightarrow \infty$. This also indicates that the delta shock waves and composite waves are physical, especially in such situations.  
\end{remark}
\subsection{Stability of the Riemann solutions with vanishing surface tension and gravity parameters}
In this section, we prove Theorem \ref{vanishing_theorem}. In particular, we analyze the behaviour of Riemann solutions from Section \ref{sec: Riemann_solutions_construction} as the parameters $\alpha$ and $\kappa$ vanish, respectively. For the sake of brevity, we focus on the first part of Theorem \ref{vanishing_theorem} ($\kappa \rightarrow 0$) as the second part of Theorem \ref{vanishing_theorem} ($\alpha\rightarrow 0$) is similar due to the linearity of the expressions of intermediate states, shock speeds, rarefaction states and the delta shock. \vspace{-0.1 cm}
\subsubsection{Riemann solutions of the system \eqref{eq: main_system}-\eqref{Riemann_data} as $\kappa\rightarrow 0$ and $\alpha=1/2$.}
In this section, we prove that the constructed Riemann solutions of \eqref{eq: main_system}-\eqref{Riemann_data} converge to the solution of the Riemann problem for \eqref{thin-film}-\eqref{Riemann_data} when the gravity parameter $\kappa\rightarrow 0$. This then proves the first part of Theorem \ref{vanishing_theorem}.\vspace{0.1 cm}\newline\\
\textbf{Case I: $h_\pm, b_\pm\in \mathcal{U}^{\circ}, ~0<h_-b_-<h_+b_+$.}\\
The Riemann solution for the system \eqref{eq: main_system} in this case is given by \eqref{J+R} with the intermediate state $(h_*, b_*)$ defined in \eqref{intermediate} and variable state $(h, b)$ inside rarefaction by \eqref{rarefaction}. Thus we directly take the limit $\kappa \rightarrow 0$ and put $\alpha=1/2$ to obtain the limit of the solution \eqref{J+R} as 
\begin{align}
    \underset{\kappa \rightarrow 0}{\lim} (h, b)(x, t)=\left\{
        \begin{array}{llll}
            (h_-, b_-),& -\infty<x<\dfrac{h_-b_- t}{2},\vspace{0.1 cm}\\
            \left(\sqrt{\dfrac{h_-h_+b_-}{b_+}}, \sqrt{\dfrac{h_-b_-b_+}{h_+}} \right),&\dfrac{h_-b_- t}{2}<{x}\leq \dfrac{3h_-b_- t}{2},\vspace{0.1 cm}\\
            \left(\sqrt{\dfrac{2x h_+}{3 b_+ t}},\sqrt{\dfrac{2xb_+}{3 h_+t }} \right),&\dfrac{3h_-b_- t}{2}\leq x\leq \dfrac{3h_+b_+ t}{2},\vspace{0.1 cm}\\
            (h_+, b_+),&\dfrac{3h_+b_+ t}{2}\leq x<+\infty,\\
        \end{array}
    \right.
\end{align}
which is exactly the Riemann solution of \eqref{thin-film}-\eqref{Riemann_data} for this choice of initial data (see \cite{sekhar2019stability}). \vspace{0.2 cm}\\
\textbf{Case II: $h_\pm, b_\pm\in \mathcal{U}^{\circ}, ~0<h_+b_+<h_-b_-$.}\\
We directly pass the limit $\kappa\rightarrow 0$ in \eqref{J+S} and obtain
\begin{align}
    \underset{\kappa \rightarrow 0}{\lim} (h, b)(x, t)=\left\{
        \begin{array}{llll}
            (h_-, b_-),& -\infty<x<\dfrac{h_-b_- t}{2},\vspace{0.1 cm}\\
            \left(\sqrt{\dfrac{h_-h_+b_-}{b_+}}, \sqrt{\dfrac{h_-b_-b_+}{h_+}} \right),&\dfrac{h_-b_- t}{2}<{x}\leq \dfrac{(h_+b_++h_-b_-+h_-b_+)t}{2},\vspace{0.1 cm}\\
            (h_+, b_+),&\dfrac{(h_+b_++h_-b_-+h_-b_+)t}{2}<x<+\infty,\\
        \end{array}
    \right.
\end{align}
which is exactly the Riemann solution of \eqref{thin-film}-\eqref{Riemann_data} for this choice of initial data (see \cite{sekhar2019stability}). \vspace{0.2 cm}\\
\textbf{Case III: $h_-\in \partial \mathcal{U}, h_+, b_\pm \in \mathcal{U}^{\circ}$.}\\
Given its equivalence to the first case, no further discussion is required for this case.\vspace{0.2 cm}\\
\textbf{Case IV: $h_+\in \partial \mathcal{U}$, $h_-, b_\pm\in \mathcal{U}^{\circ}$.}\\
For this case, let us denote the delta shock solution \eqref{pre5} as
\begin{align}
    \mathbf{U}^\kappa=(h_-, b_-)\chi_{\{x<\sigma_{\delta}t\}}+(h_+, b_+)\chi_{\{x<\sigma_{\delta}t\}}+\beta(t)\delta(x-\sigma_{\delta}t)\in \mathcal{M}_{Loc}(\mathbb{R};\mathbb{R}^2),
\end{align}
where $\chi_A$ is the characteristic function to the set $A$.

In view of the continuity of $\sigma_{\delta}$ and $\beta(t)$, we directly pass the limit $\kappa \rightarrow 0$ and $\alpha=1/2$ in \eqref{strength_delta} and \eqref{pre5} such that 
\begin{align}\label{sigma_limit}
    \underset{\kappa \rightarrow 0}{\lim} \sigma_{\delta}=\frac{h_-b_-}{2}, \quad \underset{\kappa \rightarrow 0}{\lim} \beta(t)=b_+\frac{h_-b_-t}{2}.
\end{align}
Moreover, we denote the delta shock solution of \eqref{thin-film}-\eqref{Riemann_data} as follows (see \cite{sen2020delta})
\begin{align*}
    \mathbf{U}^0=(h_-, b_-)\chi_{\{x<\frac{h_-b_-t}{2}\}}+(h_+, b_+)\chi_{\{x<\frac{h_-b_-t}{2}\}}+b_+\frac{h_-b_-t}{2}\delta\left(x-\frac{h_-b_-t}{2}\right)\in \mathcal{M}_{Loc}(\mathbb{R};\mathbb{R}^2).
\end{align*}
Thus for any $\varphi\in C_{c}^{\infty}(\mathbb{R}\times \mathbb{R}_+)$, we directly compute the limit
\begin{align}\label{limit}
    \underset{\kappa\rightarrow 0}{\lim}\langle\mathbf{U}^\kappa-\mathbf{U}^0, \varphi\rangle&=\underset{\kappa\rightarrow 0}{\lim}\displaystyle \int_{0}^{T}\int_{\mathbb{R}} \varphi\left(\left(\mathbf{U}_+-\mathbf{U}_-\right)\chi_{\left(\sigma_{\delta}t-\frac{h_-b_-t}{2}\right)}\right)\,dx\,dt\nonumber\\
    &\hspace{1 cm}+\underset{\kappa\rightarrow 0}{\lim}\displaystyle \int_{0}^{T}\int_{\mathbb{R}}\left(\varphi(\sigma_{\delta}t)\beta(t)\gamma^\kappa-\varphi\left(\frac{h_-b_-t}{2}\right)\frac{b_+h_-b_-t}{2}\gamma^0\right)\,dx\,dt,\nonumber\\
    &=\underset{\kappa\rightarrow 0}{\lim}(I_1+I_2),
\end{align}
where $\mathbf{U}_\pm=(h_\pm, b_\pm)$, $\gamma^\kappa=\sqrt{1+(\sigma_{\delta}t)^2}$ and $\gamma^0=\sqrt{1+\left(\frac{h_-b_-t}{2}\right)^2}$.

Now for $I_1$, we have the estimate
\begin{align}
   |I_1| \le \lVert \varphi\rVert_{\infty}\, \bigl|\mathbf{U}_{+}-\mathbf{U}_{-}\bigr|\, \left|\sigma_{\delta}-\frac{h_{-}b_{-}}{2}\right|\, t \le \lVert \varphi\rVert_{\infty}\, \bigl|\mathbf{U}_+-\mathbf{U}_-\bigr|\, \left|\sigma_{\delta}-\frac{h_-b_-}{2}\right|\, T. \
\end{align}
Similarly, 
\begin{align}
    |I_2|&\leq \lVert \varphi\rVert_{\infty}\left( \lVert\gamma^\kappa\rVert_{\infty}\left|\beta(t)-\frac{b_+h_-b_-t}{2}\right|
    +\frac{b_+h_-b_-t}{2} |\gamma^\kappa-\gamma^0|\right)\nonumber\\
    &\hspace{2 cm}+\lVert \beta(t)\gamma^\kappa\rVert_{\infty}\lVert \varphi(\sigma_{\delta}t)-\varphi(h_-b_-t/2)\rVert_{L^1_{Loc}}.
\end{align}
Therefore, in view of \eqref{sigma_limit} and \eqref{limit} and using the continuiy of $\varphi$, we have
\begin{align*}\label{limit_new}
    \underset{\kappa\rightarrow 0}{\lim}\langle\mathbf{U}^\kappa-\mathbf{U}^0, \varphi\rangle=0, ~\forall~t\in [0, T]. 
\end{align*}
or for all $T>0$
\[\mathbf{U^\kappa}\rightarrow \mathbf{U}^0~\text{as}~\kappa \rightarrow 0~\text{in}~C([0, T]; \mathcal{M}_{Loc}(\mathbb{R}; \mathbb{R}^2)).\]
To summarize, we can identify the limit of \eqref{pre5} as
\begin{align}
      \underset{\kappa \rightarrow 0}{\lim} (h, b)(x, t)=\left\{
        \begin{array}{lll}
            (h_-, b_-),& x<\dfrac{h_-b_- t}{2},\vspace{0.1 cm}\\
            \left(0,\dfrac{b_+h_-b_-t}{2}\delta\left(x-\frac{h_-b_-}{2}{t}\right)\right),& x=\dfrac{h_-b_- t}{2},\vspace{0.1 cm}\\
            (0, b_+),& x>\dfrac{h_-b_-t}{2},\\
         \end{array}
    \right.
\end{align}
which is the delta shock solution of the Riemann problem for \eqref{thin-film}-\eqref{Riemann_data} (see \cite{sen2020delta}).

Collecting all four cases together, the first part of Theorem \ref{vanishing_theorem} is proved. The other part of Theorem \ref{vanishing_theorem} is similar and straightforward.
\begin{remark}
Case IV shows that the strength and speed of the delta shock decrease as $\kappa\rightarrow 0$. This indicates that the gravity parameter has a direct impact on the strength and speed of delta shocks. In particular, in the context of thin film flows, peaks of concentration gradient across a delta shock get higher when the gravity parameter takes a larger value. This will be revisited again in the numerical examples in Section \ref{sec: numerics}. 
\end{remark}
\section{Stability of the Riemann solutions with perturbation in initial data}\label{sec: 3}
In what follows, we develop the solutions of the Cauchy problem \eqref{eq: main_system}-\eqref{initial_data} via nonlinear wave interactions and prove Theorem \ref{perturbed_RP}. The initial data \eqref{initial_data} can be understood as two local Riemann problems located at $(-\varepsilon, 0)$ and $(\varepsilon, 0)$ and thus the solution of \eqref{eq: main_system}-\eqref{initial_data} is constructed via nonlinear wave(s) interactions. In particular, the nonlinear wave(s) coming from the first Riemann problem may interact with the wave(s) of the second Riemann problem, and a new Riemann problem is formed at each point of interaction. For each $i$th Riemann problem, the solution is denoted by either $J_i+S_i/R_i$ or $\delta S_i$, $i=1, 2, \ldots, N, ~N\geq 2$, where “$+$” denotes “followed by” a nonlinear wave. These nonlinear interactions give rise to new emerging nonlinear classical and nonclassical hyperbolic waves. We compute all these possible points and times of interaction and provide the expression of solution curves explicitly. Moreover, during the process of interaction, the strength of each delta shock wave is also computed. All these nonlinear quantities depend on the perturbation parameter $\varepsilon$ explicitly. By analyzing each of these interaction points, nonlinear wave curves and the strength of delta shocks, we
prove that the solution of \eqref{eq: main_system} and \eqref{initial_data} given by $(h_{\varepsilon}, b_{\varepsilon})(x, t)$ converges to the Riemann solution of \eqref{eq: main_system}-\eqref{Riemann_data} as $\varepsilon\to 0$.

Depending upon the different choices of initial data, there are exactly seven possibilities for the solution of the Cauchy problem \eqref{eq: main_system}-\eqref{initial_data}. The possible cases of nonlinear wave interactions are listed as follows:
\[
\begin{aligned}
&1.\; J_1 + S_1 \text{ and } J_2 + S_2,\quad
2.\; J_1 + S_1 \text{ and } J_2 + R_2,\quad
3.\; J_1 + R_1 \text{ and } J_2 + R_2,\\[0.2cm]
&4.\; J_1 + R_1 \text{ and } J_2 + S_2,\quad
5.\; \delta S_1 \text{ and } J_2 R_2,\quad
6.\; J_1 + S_1 \text{ and } \delta S_2,\quad
7.\; J_1 + R_1 \text{ and } \delta S_2.
\end{aligned}
\]
where $J_2R_2$ is the composite wave solution of the second Riemann problem defined in \eqref{composite_wave}. Moreover, the notation for e.g. $J_1+S_1$ and $J_2+R_2$ implies that the state $(h_-, b_-)$ is connected to the state $(h_m, b_m)$ by a contact discontinuity followed by a shock wave arising from the first Riemann problem while the state the $(h_m, b_m)$ is connected to the state $(h_+, b_+)$ by a contact discontinuity followed by a rarefaction wave arising from second Riemann problem.

%These cases can be understood as an interaction of $\delta$-shock wave and elementary classical waves starting from  $(-\varepsilon, 0)$ and $(\varepsilon, 0)$. 

For the sake of brevity, here we consider only the first two cases of classical wave interactions, as the other two cases can be analyzed in a similar manner. In particular, we consider the non-trivial cases of classical wave interactions where the shock/contact or delta shock may interact with the rarefaction wave. For the delta shock wave interaction, we consider all possible cases. 
\begin{remark}
In the current state space $\mathcal{U}$, due to the nonnegativity of state variables, the interaction of two delta shock waves is not possible.   
\end{remark}
\subsection{Solutions involving classical wave interactions}
In this section, we construct the solution of the Cauchy problem \eqref{eq: main_system}-\eqref{initial_data} that involves the interaction of classical elementary waves only.
\subsubsection{Case I:  $J_1+S_1$ and $J_2+S_2$.}
For this case, the initial data \eqref{initial_data} satisfy $h_\pm, b_\pm, h_m, b_m\in \mathcal{U}^{\circ}$ such that $h_-b_- > h_mb_m > h_+b_+$ or equivalently $h_->h_m>h_+$ and $b_->b_m>b_+$. Then for a small $t>0$, the solution of the perturbed Riemann problem \eqref{eq: main_system}-\eqref{initial_data} can be described as follows:
\begin{align*}
(h_-, b_-)\quad \underset{\rightarrow}{J_1} \quad(h_1, b_1)
 \quad\underset{\rightarrow}{S_1}\quad (h_m, b_m)\quad \underset{\rightarrow}{J_2}\quad (h_2, b_2)\quad \underset{\rightarrow}{S_2}\quad (h_+, b_+), 
\end{align*}
 where $(h_1, b_1)$ and $(h_2, b_2)$ are the intermediate states between the contact discontinuity and the shock of the first and second
Riemann problems, respectively (see Figure \ref{fig:shock_interaction}). These intermediate states can be explicitly calculated using \eqref{intermediate}. 
\begin{figure}
    \centering
    \includegraphics[width=0.65\linewidth]{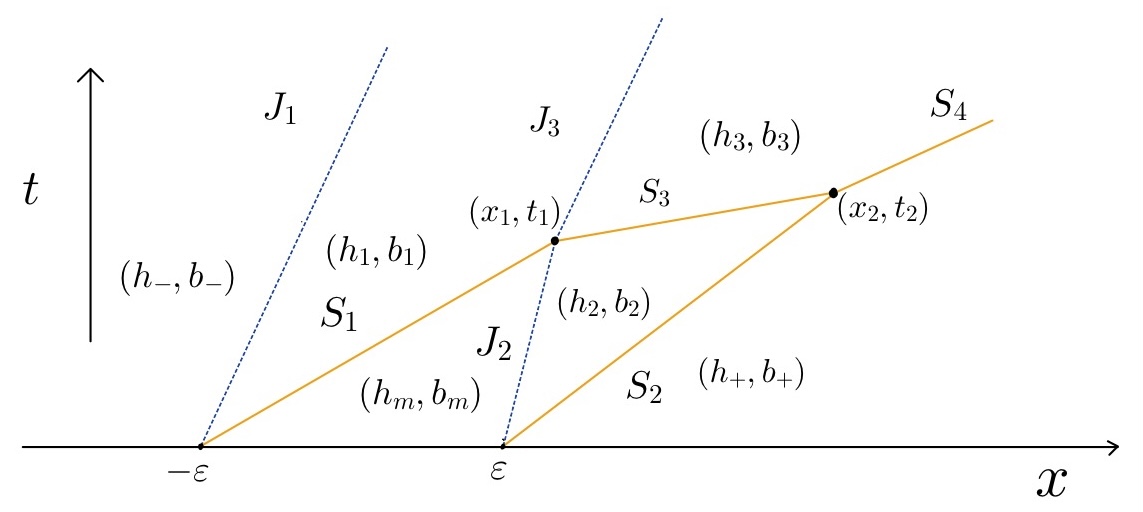}
    \caption{Solution of the perturbed Riemann problem for $h_-b_- > h_mb_m > h_+b_+$.}
    \label{fig:shock_interaction}
\end{figure}

It is easy to see that $\sigma_1>\mu_2$, where $\sigma_1$ is the speed of shock $S_1$ while $\mu_2$ is the speed of contact discontinuity $J_2$. This follows that $S_1$ interacts with $J_2$ after a finite time, say $t=t_1$. The interaction occurs at a point $(x_1, t_1)$, which is calculated by
\begin{align*}
x_1+\varepsilon= \sigma_1 t_1,\quad x_1-\varepsilon=\mu_2 t_1.
\end{align*}
From these formula, it follows that $(x_1, t_1)= \left((\sigma_1+\mu_2)\varepsilon/(\sigma_1- \mu_2), 2\varepsilon/(\sigma_1- \mu_2)\right).$

Therefore, at the interaction time $t = t_1$, the resulting configuration gives rise to a new Riemann problem. The left state is $(h_1, b_1)$ and right state is $(h_2, b_2)$. They satisfy $h_1b_1=h_-b_m\left(\dfrac{3\alpha b_-+\kappa h_-}{3\alpha b_m+\kappa h_m}\right)$ and $h_2b_2=h_mb_+\left(\dfrac{3\alpha b_m+\kappa h_m}{3\alpha b_++\kappa h_+}\right)$. Using the initial conditions $h_->h_m>h_+$ and $b_->b_m>b_+$, we obtain $h_1>h_2$ and $b_1>b_2$. Consequently, the solution of this new Riemann problem consists of a contact discontinuity $J_3$ followed by a shock wave $S_3$ with propagation speeds given by
\[
    \mu_3=h_3\left(\alpha b_3+\dfrac{\kappa h_3}{3}\right)=\mu_1,~\text{and}~
    \sigma_3=(h_2+h_3)\left(\alpha b_2+\dfrac{\kappa h_2}{3}\right)+h_3\left(\alpha b_3+\dfrac{\kappa h_3}{3}\right),
\]
where $\mu_1$ is the speed of $J_1$. Moreover, the curve of contact discontinuity $J_3$ is given by
\begin{align}
    x(t)=x_1+\left(\alpha h_- b_-+\dfrac{\kappa h_-^2}{3}\right)(t-t_1).
\end{align}
Since the propagation speeds of $J_1$ and $J_3$ are identical, the two contact discontinuities remain parallel and do not interact further for $t>t_1$.

Further, the intermediate state between $J_3$ and $S_3$ is given by $(h_3, b_3)$ such that $h_3>h_2>h_+$ and $b_3>b_2>b_+$. Therefore, for $t > t_1$, the shock curve associated with $S_3$ is given by
\begin{align}\label{eq: shock_S3}
    x(t)=x_1+\left((h_2+h_3)\left(\alpha b_2+\dfrac{\kappa h_2}{3}\right)+h_3\left(\alpha b_3+\dfrac{\kappa h_3}{3}\right)\right)(t-t_1)
\end{align}
along with the propagation speed
\begin{align*}\vspace{-0.2 cm}
    \sigma_3=(h_2+h_3)\left(\alpha b_2+\dfrac{\kappa h_2}{3}\right)+h_3\left(\alpha b_3+\dfrac{\kappa h_3}{3}\right).
\end{align*}
Further, the propagation speed of $S_2$ is
\begin{align*}\vspace{-0.2 cm}
    \sigma_2=(h_++h_2)\left(\alpha b_++\dfrac{\kappa h_+}{3}\right)+h_2\left(\alpha b_2+\dfrac{\kappa h_2}{3}\right).
\end{align*}
Since $h_3>h_2>h_+$ and $b_3>b_2>b_+$, it follows that $\sigma_3>\sigma_2$. Therefore, the two waves interact at a later time $t=t_2$, giving rise to a new Riemann problem at the interaction point $(x_2, t_2)=\left(\dfrac{\sigma_2 x_1+\sigma_2\sigma_3 t_1-\sigma_3 \varepsilon}{\sigma_2-\sigma_3}\varepsilon, \dfrac{x_1-\varepsilon-\sigma_3 t_1}{\sigma_2-\sigma_3}\right)$ with the initial data \\
$(h_3, b_3)=\left(\sqrt{\dfrac{h_-h_+(3\alpha b_-+\kappa h_-)}{3\alpha b_++\kappa h_+}}, b_+\sqrt{\dfrac{h_- (3\alpha b_-+\kappa h_-)}{h_+(3\alpha b_++\kappa h_+)}}\right)$ and $(h_+, b_+)$. 

In view of $h_3b_3>h_+b_+$, the solution of the new Riemann problem again consists of a contact discontinuity followed by a shock. Here, the contact discontinuity coincides with $J_3$, while a new shock $S_4$ develops with the propagation speed
\[\sigma_4=(h_++h_-)\left(\alpha b_++\dfrac{\kappa h_+}{3}\right)+h_-\left(\alpha b_-+\dfrac{\kappa h_-}{3}\right).\]
Clearly $\sigma_4>\mu_3$, and thus no further interaction is possible for $t>t_2$ and the solution for time $t>t_2$ takes the form
\[(h_-, b_-)\quad \underset{\rightarrow}{J_1}\quad (h_1, b_1)\quad \underset{\rightarrow}{J_3}\quad(h_3, b_3)\quad\underset{\rightarrow}{S_4}\quad(h_+, b_+),\]
where the shock curve $S_4$ is given by
\begin{align}
    x(t)=x_2+\left((h_++h_-)\left(\alpha b_++\dfrac{\kappa h_+}{3}\right)+h_-\left(\alpha b_-+\dfrac{\kappa h_-}{3}\right)\right)(t-t_2).
\end{align}
It is straightforward to observe that, as the perturbation parameter $\varepsilon\rightarrow 0$, the initial discontinuities at $x=-\varepsilon$ and $x=\varepsilon$ and the interaction points $(x_1, t_1)$ and $(x_2, t_2)$ tend to $(0, 0)$ and only the left Riemann state $(h_-, b_-)$ and the right Riemann state $(h_+, b_+)$ remains. Moreover, the curves $J_1$ and $J_3$ coincide and converge to the contact discontinuity curve $J_3$. Likewise, the shock curves $S_1$, $S_2$, $S_3$, and $S_4$ coincide and converge to the shock curve $S_4$. Therefore, the solution of the perturbed Riemann problem \eqref{eq: main_system}-\eqref{initial_data} converges to the solution of \eqref{eq: main_system}-\eqref{Riemann_data} as $\varepsilon\rightarrow 0$.

Moreover, for sufficiently large times $t>t_2$, the outcome of the interactions can be expressed as
\[(h_-, b_-)\quad\underset{\rightarrow}{J_3}\quad(h_3, b_3)\quad \underset{\rightarrow}{S_4}\quad(h_+, b_+),\]
which coincides exactly with the solution of the Riemann problem \eqref{eq: main_system}-\eqref{Riemann_data}. This demonstrates that the asymptotic behaviour of the perturbed Riemann problem is governed by the initial Riemann states as $t\rightarrow \infty$. Consequently, the solution of the Riemann problem \eqref{eq: main_system}-\eqref{Riemann_data} is globally stable under small perturbations of the initial data in this case.
\subsubsection{Case II: $J_1+S_1$ and $J_2+R_2$.}\vspace{-0.2 cm}
For this case of nonlinear interaction, the initial data \eqref{initial_data} satisfy $h_-b_-> h_mb_m>0$ and $0<h_m b_m \leq h_+b_+$. Thus, the solution of the perturbed Riemann problem \eqref{eq: main_system}-\eqref{initial_data} for sufficiently small $t>0$ is given by
\begin{align*}
(h_-, b_-)\quad \underset{\rightarrow}{J_1} \quad(h_1, b_1)
 \quad\underset{\rightarrow}{S_1}\quad (h_m, b_m)\quad \underset{\rightarrow}{J_2}\quad (h_2, b_2)\quad \underset{\rightarrow}{R_2}\quad (h_+, b_+), 
\end{align*}
Similar to the previous case, we have $\sigma_1>\mu_2$. Therefore, $S_1$ and $J_2$ interact after a finite time $t=t_1$ at the point $(x_1, t_1)=\left((\sigma_1+\mu_2)\varepsilon/(\sigma_1- \mu_2), 2\varepsilon/(\sigma_1- \mu_2)\right)$. This interaction gives rise to a new contact discontinuity $J_3$ and a new shock wave $S_3$, with the intermediate state $(h_3, b_3)$ which satisfies $h_3>h_2$ and $b_3>b_2$. 
Moreover, the propagation speeds of the shock wave $S_3$ and tail of the rarefaction wave $R_2$ are given by 
\[
    \sigma_3=(h_2+h_3)\left(\alpha b_2+\dfrac{\kappa h_2}{3}\right)+h_3\left(\alpha b_3+\dfrac{\kappa h_3}{3}\right),\quad\text{and}\quad
     \xi_2=3\alpha h_2 b_2+\kappa h_2^2,\vspace{-0.2 cm}
 \]
respectively.
Accordingly, since  $h_3>h_2$, it follows that $\sigma_3>\xi_2$. Therefore, the shock wave $S_3$ interacts with the tail of the rarefaction wave $R_2$ after a finite time at a point $(x_2, t_2)$, which can be determined from the system of equations
\begin{align*}
    x_2-x_1=\sigma_3(t_2-t_1),\quad x_2-\varepsilon=\xi_2 t_2.\vspace{-0.2 cm}
\end{align*}
and thus given by $(x_2, t_2)=\left((\xi_2 x_1+\xi_2\sigma_3 t_1-\sigma_3 \varepsilon)\varepsilon/(\xi_2-\sigma_3), (x_1-\varepsilon-\sigma_3 t_1)/(\xi_2-\sigma_3)\right)$.

\begin{figure}
    \centering
    \begin{subfigure}{0.495\linewidth}
    \includegraphics[width=\linewidth]{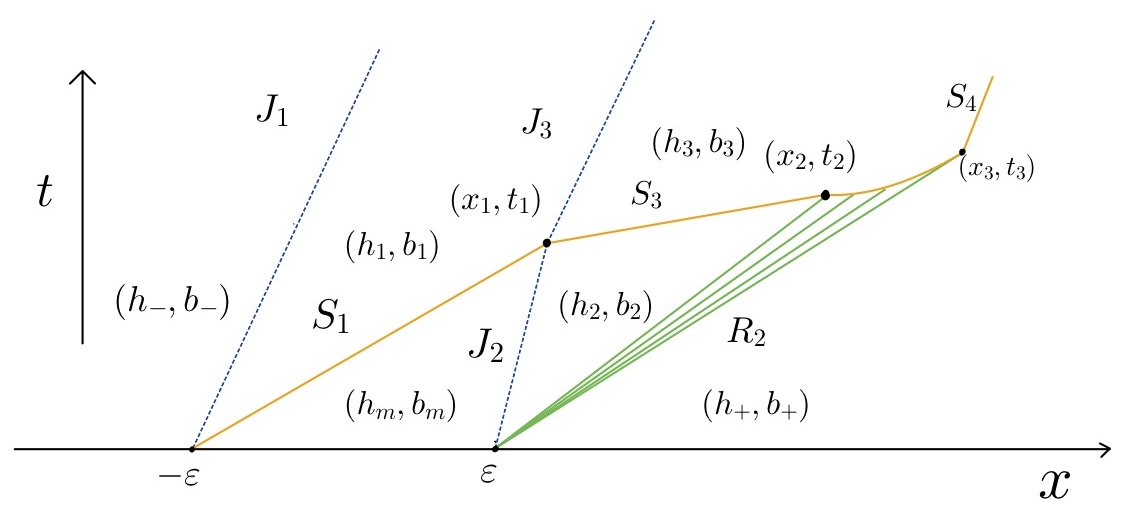}
  \caption{\centering  Solution of the \\perturbed 
  Riemann problem for \\$h_-b_-> h_mb_m$, $h_mb_m \leq h_+b_+$, \\
  $h_-b_-> h_+b_+$.}
    \label{fig:shock_subcase1}
\end{subfigure}\hfill
\begin{subfigure}{0.495\linewidth}
    \centering
    \includegraphics[width=\linewidth]{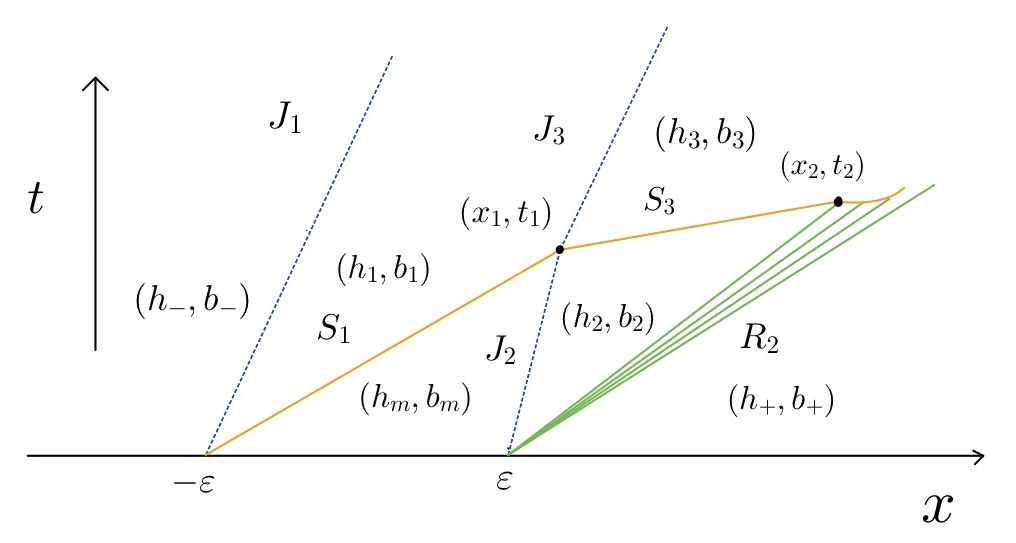}
    \caption{\centering Solution of the\\ perturbed Riemann problem for \\$h_-b_-> h_mb_m$, $h_mb_m \leq h_+b_+$, \\$h_-b_-\leq h_+b_+$.}
    \label{fig:shock_subcase2}
\end{subfigure}    
\end{figure}

The shock curve $S_3$ for $t\leq t_2$ is a straight line and is given by \eqref{eq: shock_S3}. However, for $t > t_2$, the shock wave $S_3$ enters the rarefaction $R_2$, and the shock curve can not remain a straight line anymore as the right state lies inside rarefaction $R_2$ and is no longer a constant state. The nonlinear shock curve for $t>t_2$ is expressed by
\begin{align*}
\begin{cases}
\dfrac{dx}{dt}=(h+h_3)\left(\alpha b+\dfrac{\kappa h}{3}\right)+h_3\left(\alpha b_3+\dfrac{\kappa h_3}{3}\right)\\
x-\varepsilon=(3\alpha hb+\kappa h^2)t, \\
\dfrac{b}{h}=\dfrac{b_3}{h_3}=\dfrac{b_+}{h_+},\\
x(t_2)=x_2,\quad h_2\leq h< h_3, b_2\leq b< b_3.
\end{cases}
\end{align*}
A straightforward calculation then leads to
\begin{align}\label{eq: dudt}
    \dfrac{dh}{dt}=\dfrac{(h_3^2-h^2)}{6ht}+\dfrac{h_3-h}{6t}>0.
\end{align}
and thus
\begin{align*}
    \dfrac{d^2 x}{dt^2}=\left(\left(\dfrac{2h+h_3}{3h_+}\left(3\alpha b_++\kappa h_+\right)\right)\right)\dfrac{dh}{dt}>0.
\end{align*}
This shows that the shock wave accelerates during the penetration process. Furthermore, from \eqref{eq: dudt}, we obtain
\begin{align}\label{t_expression}
    t = t_2 \exp\left( \int\limits_{h_2}^{h} 
    \displaystyle \frac {6h}{(h_3-h)(h_3+2h)}
   \,dh \right)=t_2\left(\dfrac{(h_3-h_2)^2(h_3+2h_2)}{(h_3-h)^2(h_3+2h)}\right), \quad h_2\leq h< h_3.
\end{align}
With the expression of $t$ in \eqref{t_expression}, the variable shock curve for $t>t_2$ is given by
\begin{align}
x(t)=\epsilon+t(3\alpha hb+\kappa h^2), ~h_2\leq h< h_3, ~b_2\leq b< b_3, ~x(t_2)=x_2.
\end{align}
Now, for the state $(h_3, b_3)$, two possibilities arise: either $h_3b_3\leq h_+b_+$ or $h_3b_3>h_+b_+$. We analyze these cases separately.\newline\\
\textbf{Subcase 1. ${h_3b_3>h_-b_-> h_+b_+}$.} 

In this case, the shock $S_3$ overtakes the rarefaction fan at a finite time $t_3$, which can be determined using \eqref{t_expression} by taking $h\rightarrow h_+$. At the point $(x_3, t_3)$, a new Riemann problem is formed for which $h_3b_3> h_+b_+$ and thus its solution consists of a contact discontinuity $J_3$ followed by a shock $S_4$; see \ref{fig:shock_subcase1}. Consequently, for sufficiently large times $t> t_3$, the solution of the perturbed Riemann problem can be expressed as follows:
\begin{align*}
(h_-, b_-) \quad \underset{\rightarrow}{J_3} \quad (h_3, b_3) \quad \underset{\rightarrow}{ S_4} \quad (h_+, b_+).
\end{align*}
The solution obtained here is the same as that of the unperturbed Riemann problem \eqref{eq: main_system}-\eqref{Riemann_data}. Furthermore, in the limit $\varepsilon\to 0$, all interaction points $(x_i, t_i)$ tend to $(0, 0)$ for $i=1, 2, 3$, while the nonlinear wave curves approach the shock curve $S_4$. Hence, we conclude that the solution of the Riemann problem \eqref{eq: main_system} and \eqref{Riemann_data} remains stable under small perturbations in the initial data.\newline\\
\textbf{Subcase 2. ${h_3b_3\leq h_+b_+}$.} 

For this case, it follows from the expression of $t$ in \eqref{t_expression} that the integral diverges and thus $t\rightarrow \infty$ as $h\rightarrow h_+$, which implies that the variable shock can not penetrate the rarefaction completely in a finite time; see \ref{fig:shock_subcase2}.

Therefore, the solution of the perturbed Riemann problem for $t>t_2$ can be expressed as follows
\begin{align*}
(h_-, b_-) \quad \underset{\rightarrow}{J_3} \quad (h_3, b_3) \quad \underset{\rightarrow}{ R_2} \quad (h_+, b_+).
\end{align*}
Moreover, as $\varepsilon\rightarrow 0$ the speed of the shock curve $S_1$ coincides with the back of the rarefaction wave $R$ and the two merge into the solution of the Riemann problem \eqref{eq: main_system}-\eqref{Riemann_data}. Hence, the solution of the Riemann problem \eqref{eq: main_system} and \eqref{Riemann_data} is stable under a small perturbation in initial data for this case as well.

\subsection{Solutions involving delta-shock and composite waves and their interactions}
In this section, we focus on cases that involve the interaction of delta-shock waves with classical waves and composite waves. In particular, we consider situations where certain states from the initial data \eqref{initial_data} lie on the boundary of the state space $\partial \mathcal{U}$.

\subsubsection{Case III: $\delta{S_1}$ and $J_2R_2$.}
For a given $(h_-, b_-)$, we choose $(h_m, b_m)$ and $(h_+, b_+)$ in such a way that $h_\pm, b_\pm, b_m\in \mathcal{U}^{\circ}$ and $h_m\in \partial \mathcal{U}$. Thus, for a sufficiently small time $t>0$, the solution of \eqref{eq: main_system}-\eqref{initial_data} reads as
\begin{align*}
(h_-, b_-)\quad \underset{\rightarrow}{\delta{S}_1}\quad(0, b_m)\quad \underset{\rightarrow}{J_2R_2}\quad (h_+, b_+).
\end{align*}
The speed of $\delta{S}_1$ is $\sigma_{{\delta}_1}=\alpha h_- b_-+\frac{\kappa h_-^2}{3}$ while the speed of $J_2$ is $\mu_2=0$ and thus $\sigma_{{\delta}_1}>\mu_2$. This indicates that $\delta S_1$ overtakes $J_2$ at the point $(x_1, t_1)=\left(\epsilon, \dfrac{6\varepsilon}{3\alpha h_-b_-+\kappa h_-^2}\right)$. Moreover, in view of the generalized R-H conditions, the strength of the $\delta S_1$ at $t=t_1$ is then given by $\beta(t_1)=2b_m\epsilon$. The overcompressibility condition of the delta shock fails at $x=\varepsilon$ as $h_+\neq 0$ and thus the solution can not be a new delta shock for $t>t_1$. However, we approximate the rarefaction wave $R_2$ by a set of nonphysical shock waves as in \cite{shen2018delta}. Therefore, at the point of interaction $(x_1, t_1)$, new Riemann data is formulated and defined as
\begin{align}\label{eq: NRP}
    h|_{t=t_1}=\left\{\begin{array}{ll}
        h_-, &x<x_1,\\
        h& x>x_1,
        \end{array}
    \right.,  \quad b|_{t=t_1}=
    \left\{\begin{array}{ll}
       b_-, &x<x_1,\\
        b& x>x_1,\\  
    \end{array}\right\}
    +\beta(t_1) \delta_{(x_1, t_1)},
\end{align}
In \eqref{eq: NRP}, the continuously varying right state $(h, b)$ can be obtained from
\begin{align}\label{eq: xepsilon}
\begin{cases}
    \dfrac{x-\varepsilon}{t}&=3\alpha hb+\kappa h^2,\\
    \dfrac{b}{h}&=\dfrac{b_+}{h_+},\quad 0<h<h_+, ~b_m<b<b_+,\\
    \end{cases}\vspace{0.2 cm}\nonumber\\
\hspace{-3 cm}\text{i.e.,}\quad\qquad\qquad(h, b)=\left(\sqrt{\dfrac{(x-\varepsilon) h_+}{t(3\alpha b_++\kappa h_+)}},b_+\sqrt{\dfrac{(x-\varepsilon)}{h_+t(3\alpha b_++\kappa h_+)}} \right).   
\end{align}
For $t>t_1$, $\delta S_1$ splits into a delta contact discontinuity having support on a curve $\Gamma_1$ and a shock wave having
support on some other curve $\Gamma_2$, where $\Gamma_1$ and $\Gamma_2$ are to be determined.

The solution of the local Riemann problem \eqref{eq: main_system}-\eqref{eq: NRP} can be constructed as follows:
\begin{align}\label{eq: NRP_soln}
    &h(x, t)=\begin{cases}
        h_-, &x<\sigma_{\delta_1}t,\\
        \sqrt{\dfrac{h h_-(3\alpha b_-+\kappa h_-)}{3\alpha b+\kappa h}}, & \sigma_{\delta_1}t<x<x(t),\\
      h, &x>x(t).
    \end{cases}\\
    &b(x, t)=\left\{
     \begin{array}{lll}
        b_-, &x<\sigma_{\delta_1}t,\\
        b\sqrt{\dfrac{h_-(3\alpha b_-+\kappa h_-)}{h(3\alpha b+\kappa h)}}, & \sigma_{\delta_1}t<x<x(t),\vspace{0.2 cm}\\
        b, &x>x(t).
    \end{array}
    \right\}+\beta(t_1)\left(\delta-\sigma_{\delta_1}t\right),\label{eq: NRP_soln_1}
\end{align}
where the shock curve $\Gamma_2: x=x(t)$ in the local neighbourhood of $(x_1, t_1)$ is given by
\begin{align}
x(t)=x_1+\left((h+h_-)\left(\alpha+\frac{\kappa h_-}{3}\right)+h_-\left(\alpha b_-+\frac{\kappa h_-}{3}\right)\right)(t-t_1).
\end{align}
Furthermore, $\Gamma_1$ is defined by $\Gamma_1=\{(x, t)| x(t)=\sigma_{\delta_1}(t-t_1), ~t>t_1\}$. In what follows, we prove that \eqref{eq: NRP_soln}-\eqref{eq: NRP_soln_1} is actually a weak solution for the system \eqref{eq: main_system}. First, for any $\varphi\in C_0^{\infty}(\mathbb{R}\times \mathbb{R}_+)$, it is easy to verify that \eqref{eq: NRP_soln}-\eqref{eq: NRP_soln_1} is a weak solution if $\rm{supp}~\{\varphi\cup \Gamma_1\}=\emptyset$. Alternatively, if $\rm{supp}~\{\varphi\cup \Gamma_1\}\neq \emptyset$, from standard arguments, one can easily show that \eqref{eq: NRP_soln}-\eqref{eq: NRP_soln_1} actually satisfies the weak formulation of the first equation of \eqref{eq: main_system}. However, for the second equation, a singularity can occur and needs separate attention. For the second equation, the equality
\begin{align*}
 b_t&+\left(\alpha h b^2+ \frac{\kappa hb^2}{3}\right)_x\\
 &=-\left(\alpha h_-b_-+\frac{\kappa h^2}{3}\right)\left(\left(b\sqrt{\dfrac{h_-(3\alpha b_-+\kappa h_-)}{h(3\alpha b+\kappa h)}}-b_-\right)\delta+\beta(t_1)\delta' \right)\\
&+\left(\alpha h_-b_-+\frac{\kappa h^2}{3}\right)\left(\left(b\sqrt{\dfrac{h_-(3\alpha b_-+\kappa h_-)}{h(3\alpha b+\kappa h)}}-b_-\right)\delta+\beta(t_1)\delta' \right)=0
\end{align*}
holds in the local neighbourhood of $\Gamma_1$ in the sense of distributions, in which $\delta$ and $\delta'$ are the functions of $x-\left(\alpha h_-b_-+\frac{\kappa h_-^2}{3}\right)(t-t_1)$. Thus, \eqref{eq: NRP_soln}-\eqref{eq: NRP_soln_1} satisfies the weak formulation of the second equation of \eqref{eq: main_system} in the local neighbourhood of $\Gamma_1$. Collectively, \eqref{eq: NRP_soln}-\eqref{eq: NRP_soln_1} is a weak solution of \eqref{eq: main_system}.

Therefore, we can see that the Dirac delta solution is now supported on the contact discontinuity line $x=x_1+\left(\alpha h_-b_-+\frac{\kappa h_-^2}{3}\right)(t-t_1)$ and thus is referred as the delta contact discontinuity $\delta J_3$ (see \cite{nedeljkov2008interactions}). This shows that for time $t>t_1$, the delta shock wave $\delta S_1$ decomposes into a delta contact discontinuity $\delta J_3$ and a shock $S_3$, and the state $(h_1, b_1)=\left(\sqrt{\dfrac{h_+ h_-(3\alpha b_-+\kappa h_-)}{3\alpha b_++\kappa h_+}},  b_+\sqrt{\dfrac{h_-(3\alpha b_-+\kappa h_-)}{h_+(3\alpha b+\kappa h)}}\right)$ lies between $\delta J_3$ and $S_3$. $\delta J_3$ continues to move forward with a constant speed $\left(\alpha h_-b_-+\frac{\kappa h_-^2}{3}\right)$. However, the new shock $S_3$ starts penetrating the rarefaction through the continuation of the curve $\Gamma_2$ and has a varying nonlinear curve. The variable shock wave curve during the penetration process is calculated by
\begin{align*}
\begin{cases}
\dfrac{dx}{dt}=(h+h_-)\left(\alpha b+\frac{\kappa h}{3}\right)+h_-\left(\alpha b_-+\frac{\kappa h_-}{3}\right),\vspace {0.2 cm}\\
\dfrac{x-\varepsilon}{t}=3\alpha hb+\kappa h^2, 0\leq h\leq h_+, b_m\leq b\leq b_+.
\end{cases}
\end{align*}
\begin{figure}
    \centering
    \begin{subfigure}{0.495\linewidth}
    \includegraphics[width=\linewidth]{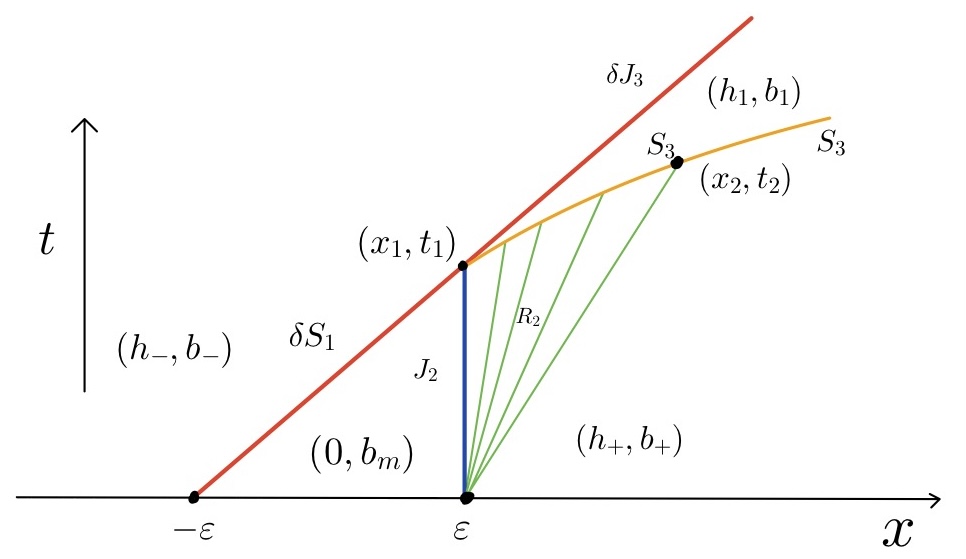}
    \caption{ \centering Solution of the \\perturbed 
  Riemann problem for\\
  $h_m=0$, $0<h_+b_+<h_-b_-$.}
    \label{fig:subcase_1_delta}
    \end{subfigure}\hfill
     \begin{subfigure}{0.495\linewidth}
    \includegraphics[width=\linewidth]{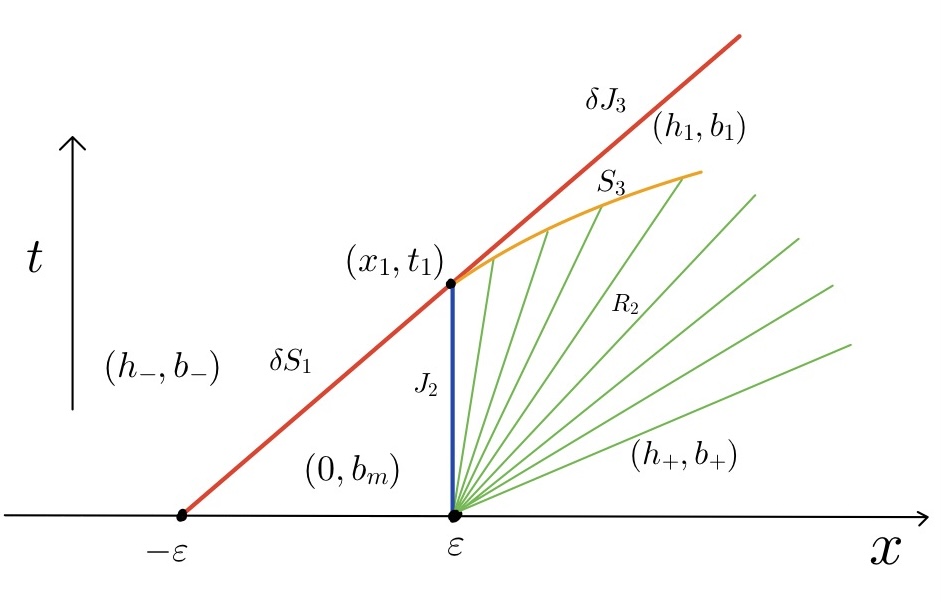}
    \caption{\centering Solution of the \\perturbed 
  Riemann problem for\\ $h_m=0$, $0<h_-b_-<h_+b_+$.}
    \label{fig:subcase_2_delta}
    \end{subfigure}
\end{figure}
It is clear that at time $t=t_1$, shock speed coincides with the speed of $\delta J_3$ and thus $\delta J$ is tangential with $S_3$ at $(x_1, t_1)$. By a similar calculation as in Case II, one can see that $\frac{d^2 x}{dt^2}<0$, which indicates that the shock wave slows down during the process of penetration. Based on the comparsion between the values of $h_-, b_-$ and $h_+, b_+$, we divide the discussion into following two cases.\vspace{0.2 cm}\\

\textbf{Subcase 1. $0<h_+b_+<h_-b_-$.}

In this subcase, $S_3$ penetrates $R_2$ completely at a point $(x_3, t_3)$ which can be obtained by solving the following two equations
\begin{align}
    x-x_1&=(h+h_-)\left(\alpha b_-+\frac{\kappa h_-}{3}\right)+h_-\left(\alpha b_-+\frac{\kappa h_-}{3}\right)(t-t_1)\\
    x-\varepsilon&= 3\alpha h_+ b_++\kappa h_+^2
\end{align}
For $t>t_3$, the shock wave propagates with an invariant speed. Letting $\varepsilon\rightarrow 0$, one can easily see that the limit of the solution \eqref{eq: main_system}-\eqref{initial_data} is the contact discontinuity $J$ followed by a shock wave $S$; see \ref{fig:subcase_1_delta}.\newline\\

\textbf{Subcase 2. $0<h_-b_-<h_+b_+$.}

In this subcase, $S_3$ is unable to cancel  $R_2$ completely and has the straight line $x(t)=\varepsilon+\left(3\alpha h_-b_-+\kappa h_-^2\right)t$ as its asymptotic line; see \ref{fig:subcase_2_delta}.

\subsubsection{Case IV: $J_1+S_1$ and $\delta{S_2}$.}
Here, we choose the initial data \eqref{initial_data} such that $0<h_mb_m<h_-b_-$ and $h_+=0$.
Therefore, for a sufficiently small time $t$, the solution structure is then given by
\begin{align*}
(h_-, b_-)\quad \underset{\rightarrow}{J_1}\quad(h_1, b_1)\quad \underset{\rightarrow}{S_1}\quad (h_m,b_m)\quad \underset{\rightarrow}{\delta{S_2}}\quad(0,b_+).
\end{align*}
In analogy with the first case, the propagation speed $\sigma_1$ of the shock wave ($S_1$) is greater than the speed $\sigma_{\delta_2}$ of the delta shock wave ($\delta{S}_2$). It follows that $S_1$ eventually overtakes $\delta S_2$ and the interaction occurs at a finite time $t=t_1$ at a point $(x_1, t_1)= \left((\sigma_1+\sigma_{\delta_2})\varepsilon/(\sigma_1- \sigma_{\delta_2}), 2\varepsilon/(\sigma_1- \sigma_{\delta_2})\right)$ in the $x-t$ plane. Using the generalized Rankine-Hugoniot conditions \eqref{GRH}, the strength of the delta shock $\delta{S_1}$ at the interaction point $(x_1, t_1)$ is given by
\begin{align}\label{sec2}
\beta(t_1)=b_+\sigma_{\delta_2}t_1.
\end{align}
\begin{figure}
    \centering
    \includegraphics[width=0.65\linewidth]{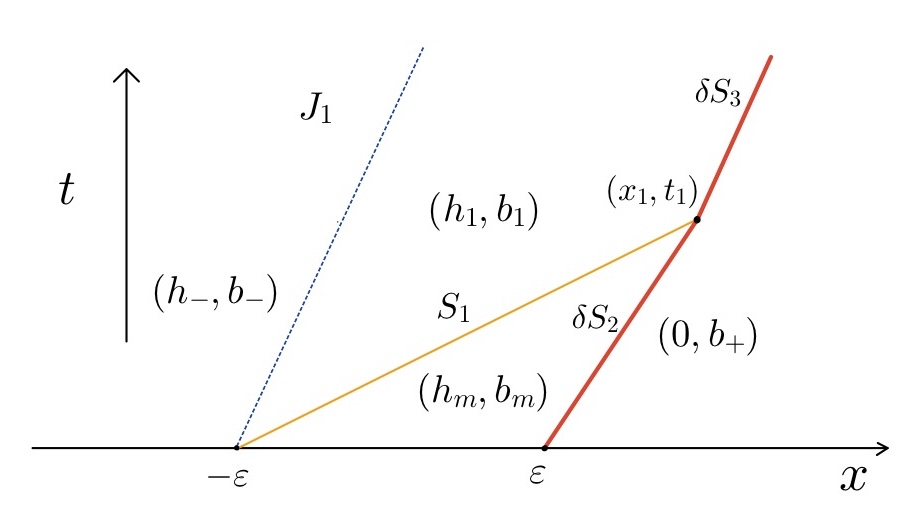}
    \caption{Solution of the perturbed Riemann problem for $0=h_+<h_mb_m<h_-b_-$.}
    \label{fig:placeholder}
\end{figure}
At the point of interaction $(x_1, t_1)$, a new Riemann problem is formed with the initial data of the form
\begin{align*}
h|_{t=t_1}= \left\{
        \begin{array}{ll}
            h_1,&x<x_1,\\
            0,&x>x_1,\\
         \end{array}
    \right.,~~ b|_{t=t_1}= \left\{
        \begin{array}{ll}
            b_1,&x<x_1,\\
            b_+,&x>x_1,\\
         \end{array}
    \right\}
+\beta(t_1)\delta_{(x_1, t_1)},
\end{align*}
where $(h_1, b_1)$ is the intermediate state between $J_1$ and $S_1$ and is obtained by \eqref{intermediate}.

Due to the overcompressibility conditions \eqref{eq: GEC}, a new $\delta$-shock $\delta {S_3}$ is generated after the interaction from the point $(x_1, t_1)$, which can be described as
\begin{align*}
h(x, t)= \left\{
        \begin{array}{ll}
            h_1,&x<x(t),\\
            0,&x>x(t),\\
         \end{array}
    \right.,~~ b(x, t)= \left\{
        \begin{array}{ll}
            b_1,&x<x(t),\\
            b_+,&x>x(t),\\
         \end{array}
    \right\}
+ \beta(t)\delta{(x-x(t))},
\end{align*}
where delta function have support on the curve $\Gamma: x(t)=x_1+\sigma_{\delta_3}(t-t_1)$ with 
\[\sigma_{\delta_3}=\alpha h_1b_1+\dfrac{\kappa h_1^2}{3}=\alpha h_-b_-+\dfrac{\kappa h_-^2}{3}=\mu_1.\]
Since the propagation speed of the delta shock $\delta{S}_3$ is equal to the speed of the contact discontinuity $J_1$, the delta shock $\delta{S}_3$ does not interact with $J_1$ for $t>t_1$. Moreover, its strength is given by $\beta(t)=b_+\sigma_{\delta_3}t, \quad t>t_1.$

It is clear that $(x_1, t_1)\rightarrow (0, 0)$, $\beta(t_1)\rightarrow 0$ as $\varepsilon \rightarrow 0$ and thus the solution of the perturbed Riemann problem converges to the delta shock solution of the Riemann problem. This indicates that the solution of the Riemann problem is stable under the perturbation of initial data. Moreover, for $t>t_1$, the solution structure of \eqref{eq: main_system}-\eqref{initial_data} is $\delta S_3$, which is the same as the delta shock solution of the Riemann problem \eqref{eq: main_system}-\eqref{Riemann_data}, which proves that the asymptotic behaviour of the perturbed Riemann problem is governed by the Riemann states for sufficiently large time.

\subsubsection{Case V:  $J_1+R_1$ and $\delta{S_2}$.}
For this case, we choose the initial data \eqref{initial_data} such that $0<h_-b_-<h_mb_m$ and $h_+=0$.
 For a sufficiently small $t>0$, the solution structure is then given by
\begin{align*}
(h_-, b_-)\quad \underset{\rightarrow}{J_1}\quad(h_1, b_1)\quad \underset{\rightarrow}{R_1}\quad (h_m,b_m)\quad \underset{\rightarrow}{\delta{S_2}}\quad(0,b_+).
\end{align*}
Since the speed of the $\delta$-shock wave of the second Riemann problem $\sigma_{\delta_2}=\alpha h_mb_m+\kappa h_m^2/3$ is less than the speed of the wave front of the rarefaction wave $\xi_1=3\alpha h_mb_m+\kappa h_m^2$ of the first Riemann problem, it follows that $\delta{S}_2$ interact with $R_1$ after a finite time $t=t_1$ and the interaction take place at a point $\left(x_1, t_1\right)=\left(2\varepsilon, \dfrac{3\varepsilon}{3\alpha h_mb_m+\kappa h_m^2}\right)$ in the $x-t$ plane. 

\begin{figure}
    \centering
    \includegraphics[width=0.65\linewidth]{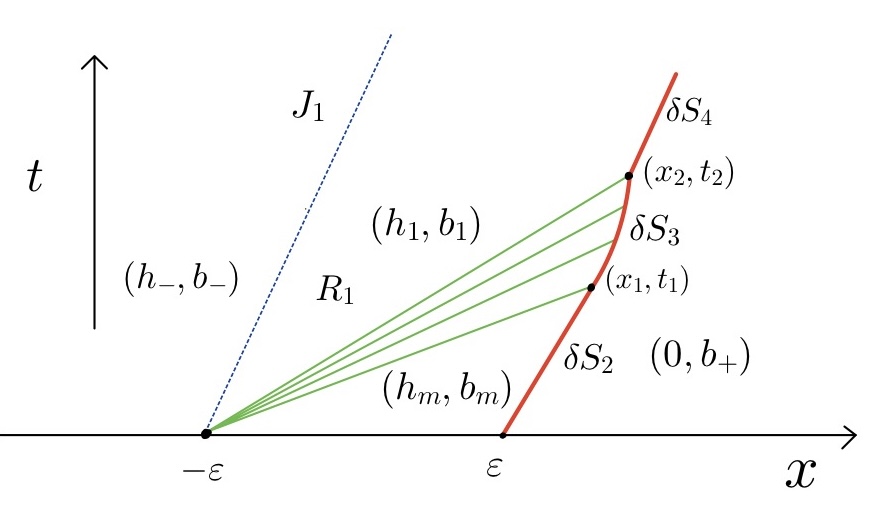}
    \caption{Solution of the perturbed Riemann problem for $0=h_+<h_-b_-<h_mb_m$.}
    \label{fig:placeholder}
\end{figure}
\begin{comment}
The interaction point $\left(x_1, t_1\right)=\left(\frac{\varepsilon(\xi_a+\sigma_{\delta_1})}{\xi_a-\sigma_{\delta_1}}, \dfrac{2\varepsilon}{\xi_a-\sigma_{\delta_1}}\right)$ can be calculated as follows
\begin{align*}
 \left\{
       \begin{array}{ll}
        x_1+\varepsilon=\xi_a t_1,\\
        x_1-\varepsilon=\sigma_{\delta_1} t_1,
       \end{array}
  \right.
\end{align*}
which implies that
\begin{align}\label{sec3}
\left(x_1, t_1\right)=\left(\frac{\varepsilon(\xi_a+\sigma_{\delta_1})}{\xi_a-\sigma_{\delta_1}}, \dfrac{2\varepsilon}{\xi_a-\sigma_{\delta_1}}\right).
\end{align}
\end{comment}
The strength of $\delta{S_2}$ at $(x_1, t_1)$ is given by
\begin{align}\label{sec4}
\beta(t_1)=b_+\sigma_{\delta_2}t_1=b_+ \varepsilon.
\end{align}
For $t>t_1$, the delta shock wave $\delta S_2$ penetrates the rarefaction wave $R_1$. During this penetration process, a new delta shock wave $\delta{S}_3$ is generated due to the overcompressibility conditions \eqref{eq: GEC} ($h>0=h_+$). The wave $\delta{S}_3$ is the solution of the local Riemann problem with a fixed right state $(h_+, b_+)=(0, b_+)$ and a variable left state $(h, b)$ varying continuously within the rarefaction fan $R_1$. We use $\Gamma_1:\{(x(t), t): t\geq{t_1}\}$ with $x(t)$ to express the curve $\delta{S}_3$. The varying curve can be obtained using
\begin{align}\label{varying_shock}
\begin{cases}
\dfrac{dx}{dt}=\alpha hb+\dfrac{\kappa h^2}{3},\vspace{0.1 cm}\\
\dfrac{x+\varepsilon}{t}=3\alpha hb+\kappa h^2, \\
\dfrac{b}{h}=\dfrac{b_m}{h_m},\\
x(t_1)=x_1,\quad h_m\leq h\leq h_-, b_m\leq b\leq b_-.
\end{cases}
\end{align}
From the first two equations, we have
\begin{align}
    \dfrac{dx}{dt}=\dfrac{x+\varepsilon}{3t}
\end{align}
Therefore, the delta shock curve is given by
\begin{align}\label{delta_shock_new}
    x(t)=\left(9\varepsilon^2 (3\alpha h_mb_m+\kappa h_m^2) t\right)^{1/3}-\varepsilon.
\end{align}
Moreover, a computation analogous to that in Case 2 shows that  $\dfrac{dh}{dt}<0$, which in turn implies $\dfrac{d^2 x}{dt^2}<0$. Therefore, the propagation speed of the delta shock $\delta S_3$ decreases during the penetration process. 

For $t>t_1$, the strength of the delta shock wave $\delta S_3$ is obtained from the equation $\frac{d\beta(t)}{dt}=b_+\left(\alpha hb+\dfrac{\kappa h^2}{3}\right)$ with initial condition $\beta(t_1)$. Using the second equation from \eqref{varying_shock} and \eqref{delta_shock_new}, we have
\begin{align}
    \dfrac{d\beta}{dt}=\dfrac{b_+}{3}\left(\dfrac{9\varepsilon^2 (3\alpha h_mb_m+\kappa h_m^2)}{t^2}\right)^{1/3}.
\end{align}
Therefore, the strength of $\delta S_3$ for $t>t_1$ is given by
\begin{align}\label{delta_shock_strength_new}
    \beta(t)=b_+(9\varepsilon^2(3\alpha h_mb_m+\kappa h_m^2))^{1/3}(t^{1/3}-(t_1)^{1/3})+\beta(t_1).
\end{align}
Finally, $\delta{S}_3$ cuts the entire rarefaction $R_1$ in a finite time $t=t_2$ and ends at the interaction point $(x_2, t_2)$, which is obtained from \eqref{delta_shock_new} along with the relation of wave back of rarefaction $R_1$ such that $\frac{x+\varepsilon}{t}=3\alpha h_-b_-+\kappa h_-^2$. A simple computation then yields 
\begin{align}
    t_2=\dfrac{(9\varepsilon^2 (3\alpha h_mb_m+\kappa h_m^2)^{1/3}}{(3\alpha h_-b_-+\kappa h_-^2)},\quad x_2=\left(9\varepsilon^2 (3\alpha h_mb_m+\kappa h_m^2) t_2\right)^{1/3}-\varepsilon.
\end{align}
At point $(x_2, t_2)$, we again have a new Riemann problem for which the solution for $t>t_2$ is given by a new delta shock $\delta{S}_4$ due to overcompressibility conditions. The propagation speed of $\delta{S}_4$ is $\alpha h_-b_-+\frac{\kappa h_-^2}{3}$, which coincides with the speed of $J_1$. Consequently, no further interaction between $\delta{S}_4$ and $J_1$ is possible. The strength of $\delta{S}_4$ is given by $\beta(t)=\beta(t_2)+b_+\left(\alpha h_-b_-+\frac{\kappa h_-^2}{3}\right)(t-t_2)$, where $\beta(t_2)$ is given by \eqref{delta_shock_strength_new}.

Moreover, for $t>t_2$, the solution of the perturbed Riemann problem can be expressed as follows
\begin{align*}
(h_-, b_-) \quad \underset{\rightarrow}{\delta{S}_4}\quad (h_+, b_+).
\end{align*}

Also, as $\varepsilon\to 0$, all the interaction points $(x_i, t_i)\rightarrow (0, 0), ~i=\{1, 2\}$ and strength $\beta(t)\rightarrow b_+(\alpha h_-b_-+\kappa h_-^2/3)t$, which is the strength of delta shock solution of the Riemann problem \eqref{eq: main_system}-\eqref{Riemann_data}. Moreover, the speed of the contact discontinuity ($J_1$) coincides with the delta shock wave $(\delta{S}_4)$ and the two merge into the delta shock solution of the Riemann problem \eqref{eq: main_system}-\eqref{Riemann_data}. Hence, the solution of the Riemann problem \eqref{eq: main_system} and \eqref{Riemann_data} is stable under such a small perturbation in initial data for this case as well.

Collecting all the cases (I-V) together, Theorem \ref{perturbed_RP} is proved.
\section{Numerical examples}\label{sec: numerics}
In this section, we apply the Godunov scheme based on the Riemann solution from Section \ref{Sec: Riemann} for the cases when there are only classical waves and a finite-difference Lagrangian-Eulerian scheme from \cite{abreu2019fast, de2021interaction} for the cases that include delta shocks, as it has been proven to capture sharp discontinuities much better. Using these schemes, we validate our analytical results. 
\subsection{Vanishing gravity limit of the Riemann solutions}
In this section, we discuss the behaviour of the Riemann solutions of the system \eqref{eq: main_system} as $\kappa\rightarrow 0$. 
\begin{example}[$J+R$]\label{ex:J+R}
Consider the  Riemann data 
\begin{align}
    (h, b)^\top (x, 0)=\begin{cases}
        (1.24, 0.90)^\top,\quad x<0,\\
        (1.5, 1.56)^\top, ~\quad x>0.
    \end{cases}
\end{align}
The solution for this case is a contact discontinuity followed by a rarefaction wave. We use the constructed Riemann solutions to develop a Godunov solver for this case and plot the numerical solutions at time $t=1.00$ with $\Delta x=5.33\times 10^{-3}$. The asymptotic convergence towards the Riemann solutions of \eqref{thin-film}-\eqref{Riemann_data} is observed as $\kappa\rightarrow 0$; see Figure \ref{fig: J+R_new}.

\begin{figure}
    \centering
    \includegraphics[width=\linewidth]{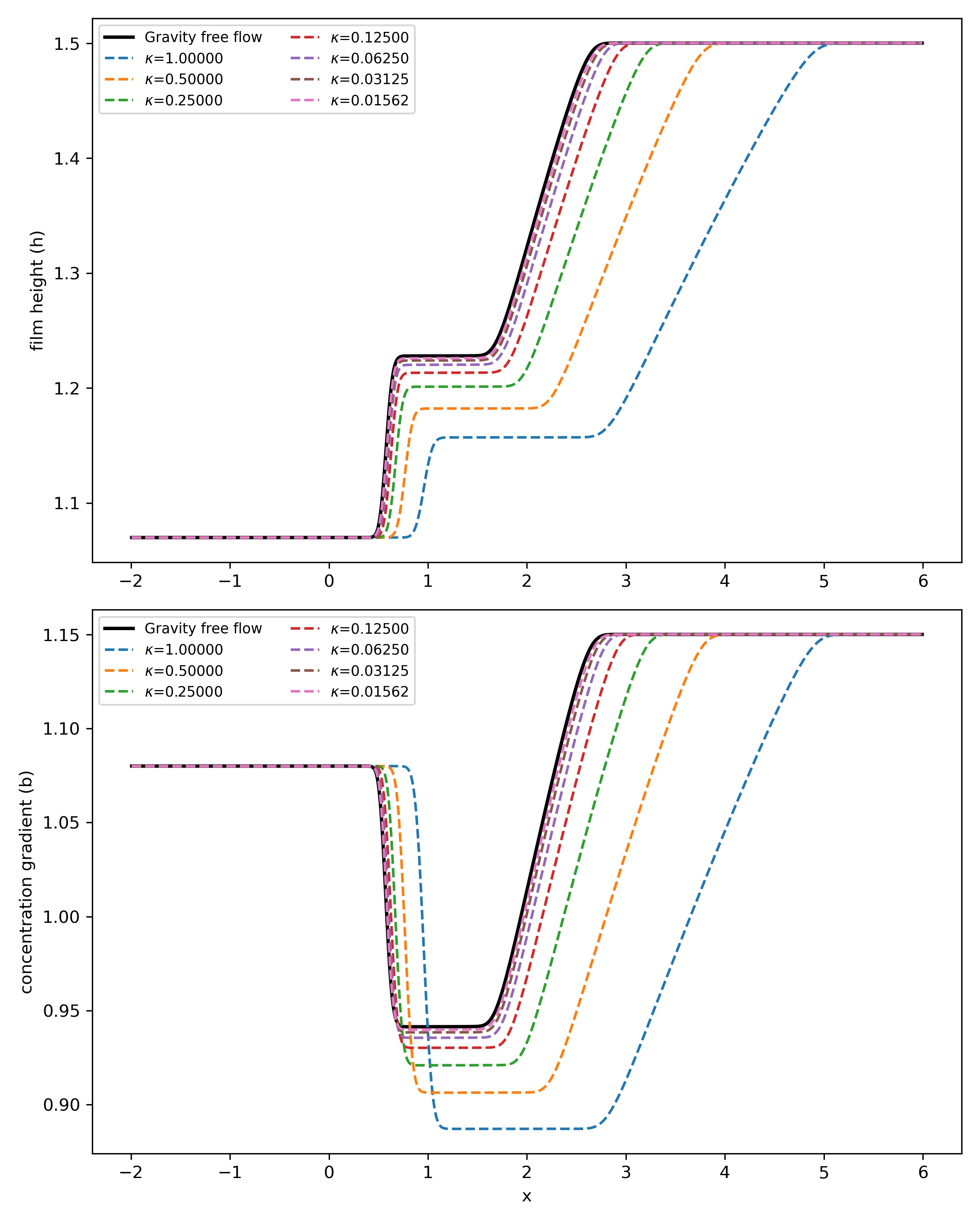}
    \caption{Convergence of Riemann solutions of \eqref{eq: main_system}-\eqref{Riemann_data} to the solutions of \eqref{thin-film}-\eqref{Riemann_data} as $\kappa\rightarrow 0$ ($J+R$).}
    \label{fig: J+R_new}
\end{figure}
\begin{figure}
        \centering
        \includegraphics[width=\linewidth]{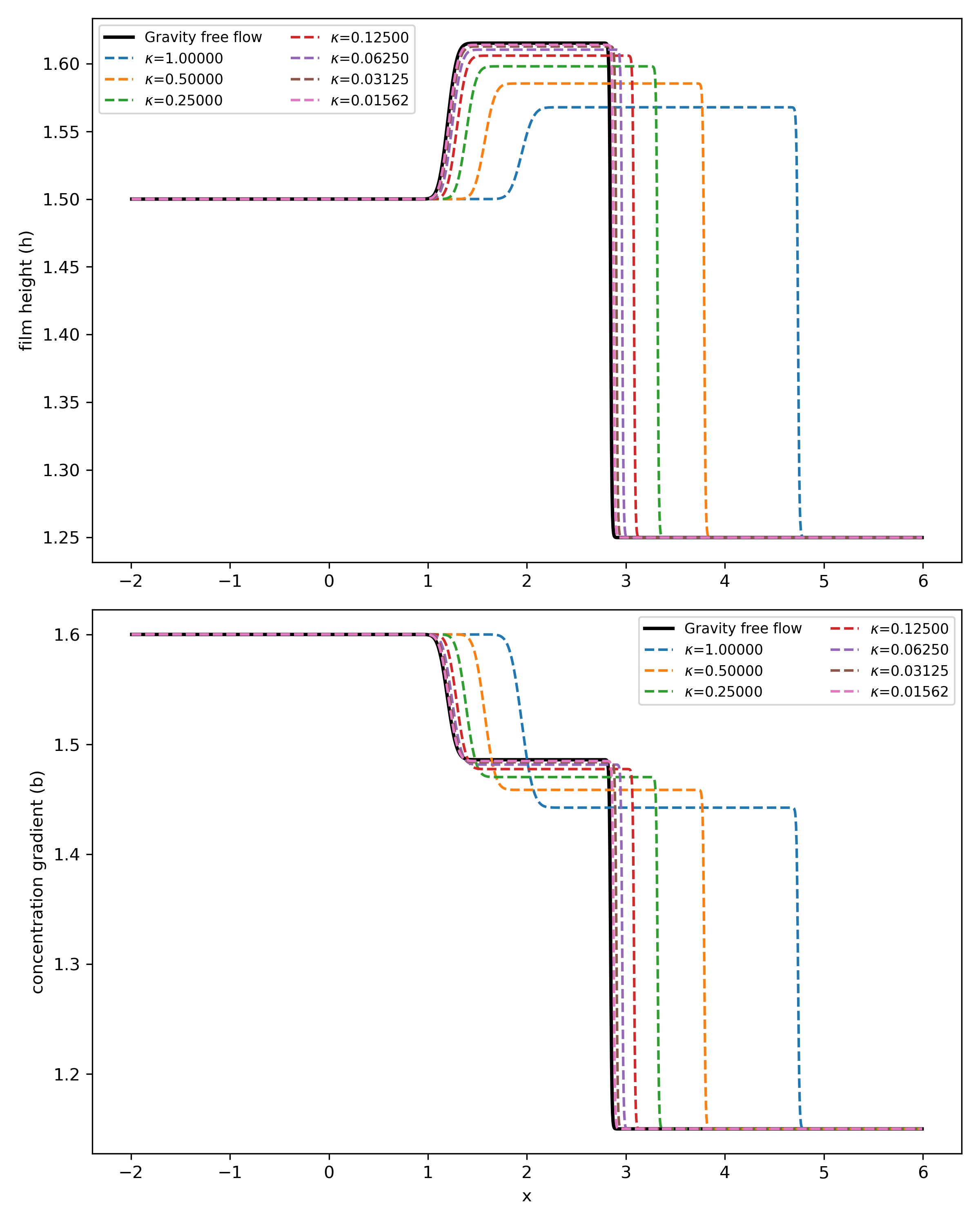}
        \caption{Convergence of Riemann solutions of \eqref{eq: main_system}-\eqref{Riemann_data} to the solutions of \eqref{thin-film}-\eqref{Riemann_data} as $\kappa\rightarrow 0$ ($J+S$).}
        \label{fig: J+S_new}
\end{figure}
\end{example}
\begin{example}[$J+S$]\label{ex:J+S}
In this example, we consider the following Riemann data 
\begin{align}
    (h, b)^\top (x, 0)=\begin{cases}
        (1.5, 1.6)^\top,~\qquad x<0,\\
        (1.25, 1.15)^\top, ~\quad x>0.
    \end{cases}
\end{align}
The solution for this case is a contact discontinuity followed by a shock wave. We again use the Godunov solver for this case and plot the numerical solutions at time $t=1.00$ with $\Delta x=5.33\times 10^{-3}$. The asymptotic convergence is again observed as $\kappa\rightarrow 0$; see Figure \ref{fig: J+S_new}.
\end{example}
\begin{example}[Delta shock ($\delta S$)]
In this example, we consider the following Riemann data 
\begin{align}
    (h, b)^\top (x, 0)=\begin{cases}
        (2.9, 1.70)^\top,~~\quad\qquad x<0,\\
        (0.0000001, 5.56)^\top, \quad x>0.
    \end{cases}
\end{align}
\begin{figure}
    \centering
    \includegraphics[width=\linewidth]{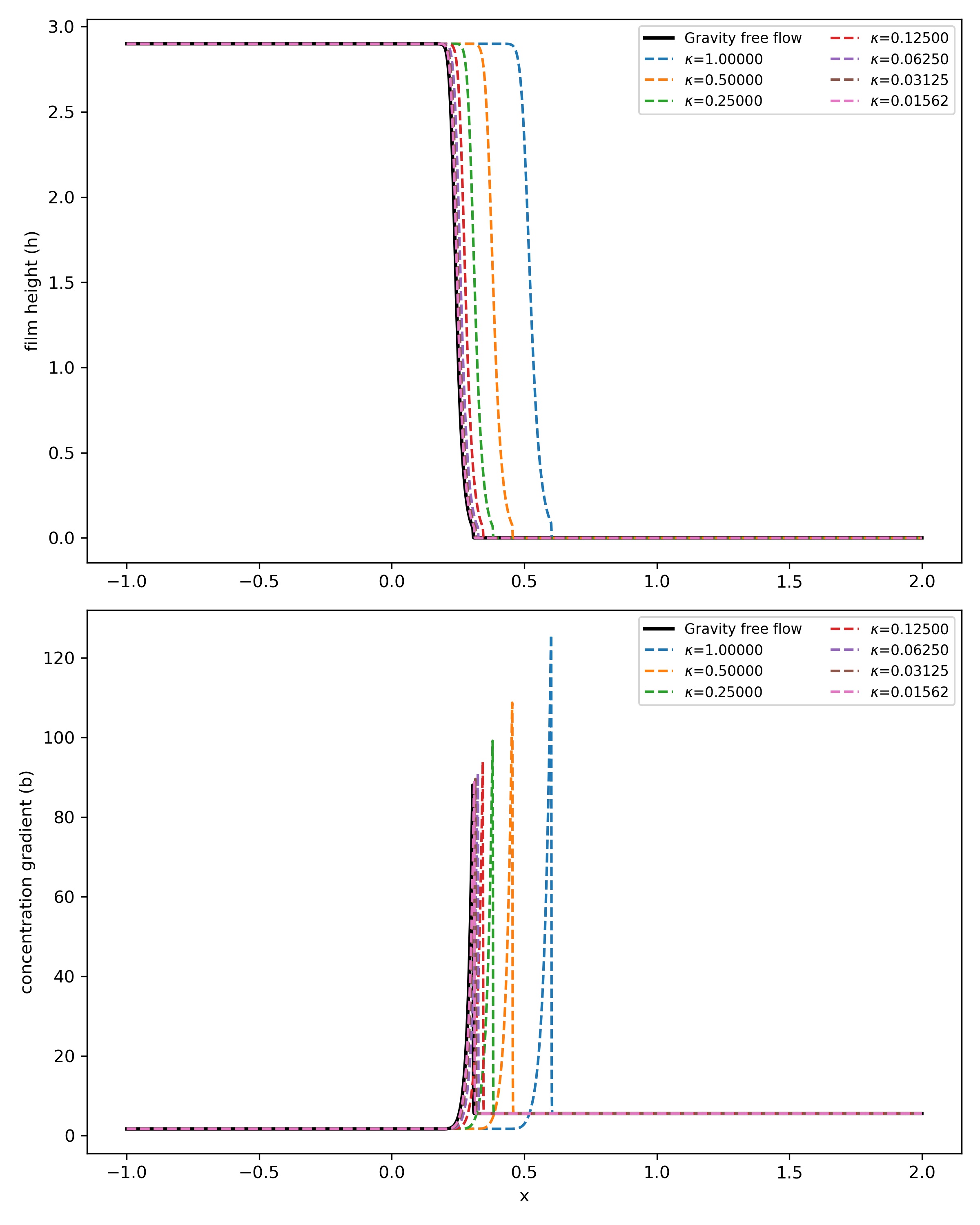}
    \caption{Convergence of delta shock solution of \eqref{eq: main_system}-\eqref{Riemann_data} to the delta shock solution of \eqref{thin-film}-\eqref{Riemann_data} as $\kappa \rightarrow 0$.}
    \label{fig:deltaS}
\end{figure}
The solution for this case is a single delta shock wave, which can be observed in Figure \ref{fig:deltaS}. We use the Lagrangian-Eulerian scheme from \cite{abreu2019fast} to plot the numerical solutions at time $t=0.10$ with $\Delta x=1\times 10^{-4}$. It is interesting to observe the effect of the gravity parameter $\kappa$ on the strength of the delta shock wave, which resonates with the expression of the strength of the delta shock and its dependency on $\kappa$.
\end{example}
\subsection{Perturbed Riemann problem}
\begin{example}
In this example, we perturb the initial data from Example \ref{ex:J+R} and include a perturbation parameter $\varepsilon$ in the initial discontinuity. In particular, we consider the following initial data
 \begin{align}\label{J+R_per_data}
    (h, b)^\top (x, 0)=\begin{cases}
        (1.24, 0.90)^\top,~~~~~\quad -\infty <x<-\varepsilon,\\
        (0.75, 1.25)^\top, ~~\qquad -\varepsilon<x<\varepsilon, \\
        (1.5, 1.56)^\top, ~~~~~\qquad \varepsilon<x<+\infty.
    \end{cases}
\end{align}
We use the Godunov solver again, and plot the numerical solutions at time $t=1.0$ and $t=15.0$ in Figure \ref{fig:J+R_t=1} and Figure \ref{fig:J+R_t=15}, respectively. It can be observed that at time $t=1.0$, the solution of the perturbed Riemann problem is the interaction of $J_1+S_1$ and $J_2+R_2$, which converges to the solution of the Riemann problem as $\varepsilon\rightarrow 0$. However, for a long time, the solution of the perturbed Riemann problem behaves as the solution of the Riemann problem even for large values of $\epsilon$, which indicates that the asymptotic behaviour $(t\rightarrow \infty)$ of the perturbed Riemann problem is governed by the underlying Riemann problem only.
\begin{figure}
\centering
    \includegraphics[width=\linewidth]{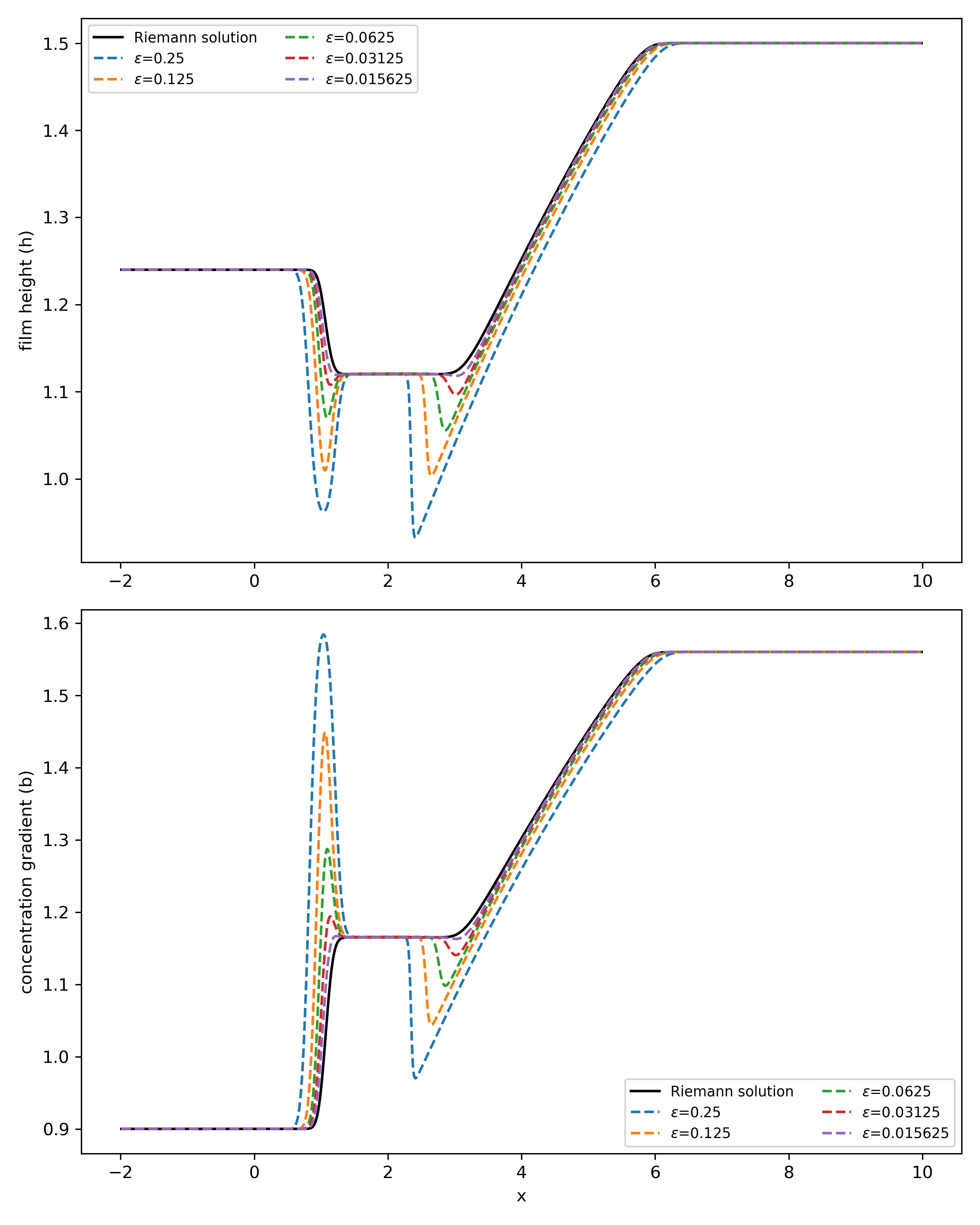}
    \caption{Solution of perturbed Riemann problem with initial data \eqref{J+R_per_data} at $t=1.0$ \\(Interaction of $J_1+S_1$ and $J_2+R_2$).}
    \label{fig:J+R_t=1}
\end{figure}
    \begin{figure}
        \centering
        \includegraphics[width=\linewidth]{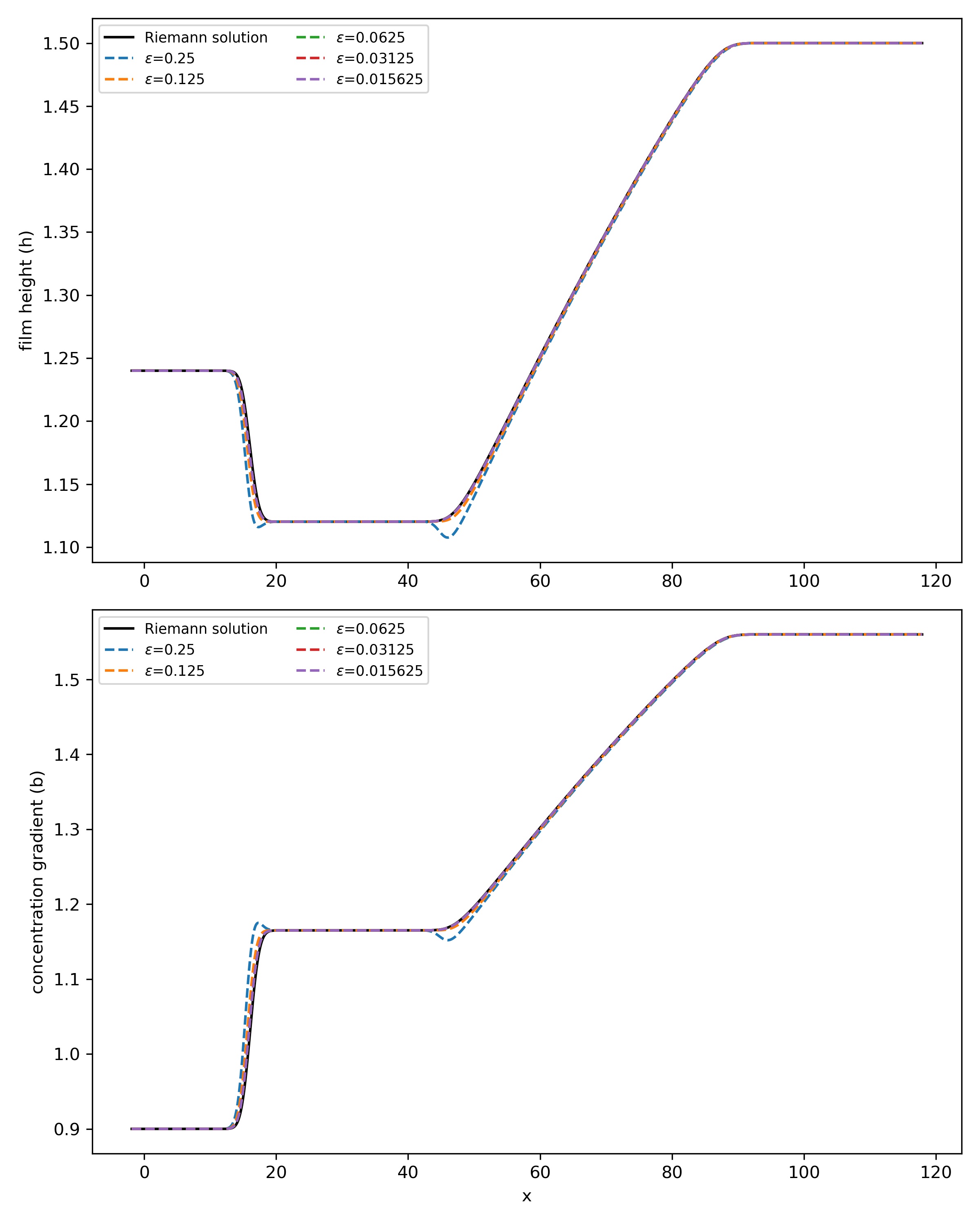}
        \caption{Long time behaviour of the solutions with initial data \eqref{J+R_per_data} at $t=15.0$.}
        \label{fig:J+R_t=15}
    \end{figure}
 \end{example}
 \begin{example}
In this example, we perturb the initial data from Example \ref{ex:J+S} and include a perturbation parameter $\varepsilon$. In particular, we consider the following initial data
     \begin{align}\label{ex:J+S_per}
    (h, b)^\top (x, 0)=\begin{cases}
         (1.5, 1.6)^\top,~~~~\qquad -\infty <x<-\varepsilon,\\
        (0.95, 1.62)^\top, ~~\qquad -\varepsilon<x<\varepsilon, \\
        (1.25, 1.15)^\top, ~~\qquad \varepsilon<x<+\infty.
    \end{cases}
\end{align}

\begin{figure}
        \centering
        \includegraphics[width=\linewidth]{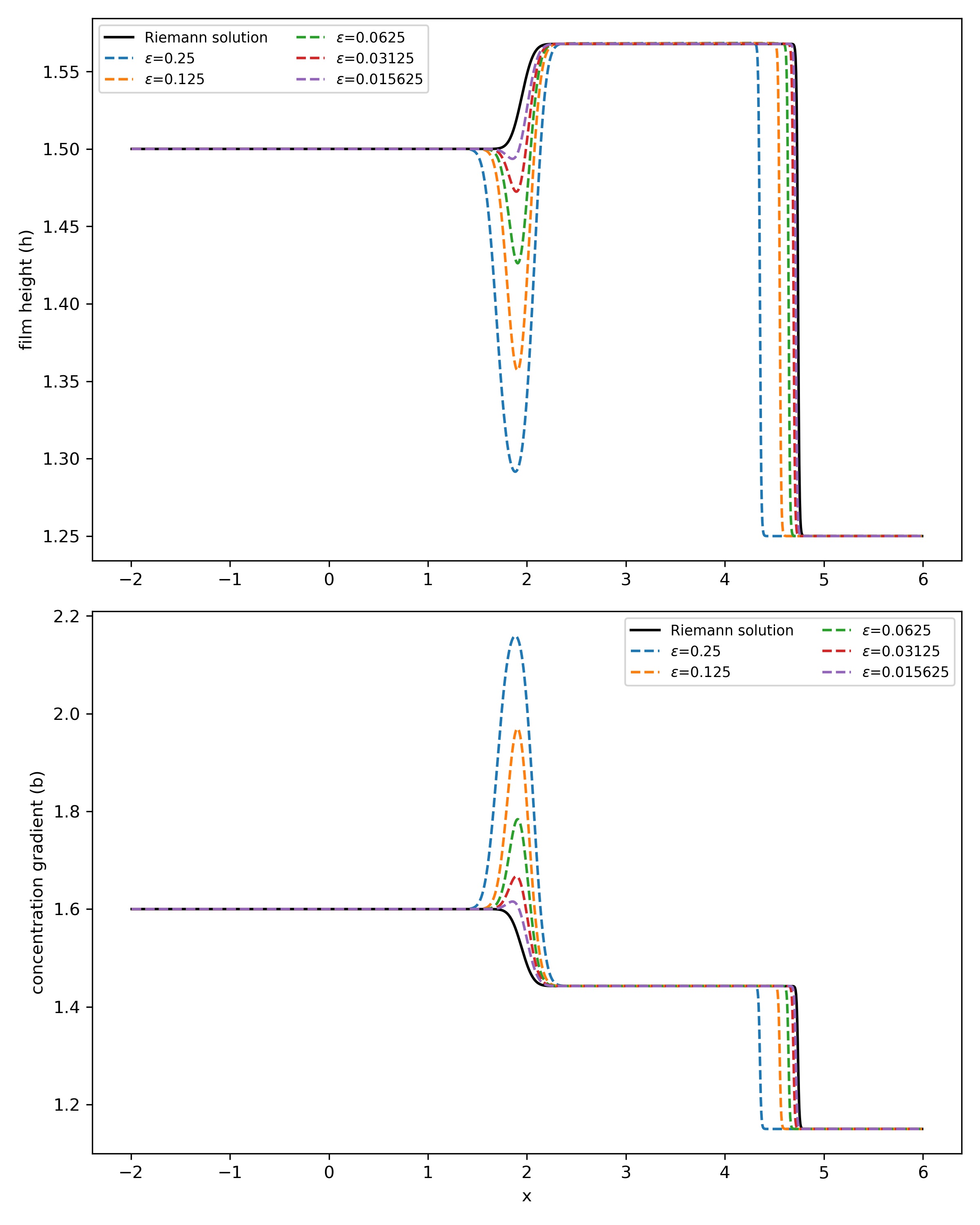}
        \caption{Solution of perturbed Riemann problem with initial data \eqref{ex:J+S_per} at $t=1.0$\\ (Interaction of $J_1+S_1$ and $J_2+S_2$).}
        \label{fig:J+S_per_t=1}
\end{figure}
\begin{figure}
    \centering
    \includegraphics[width=\linewidth]{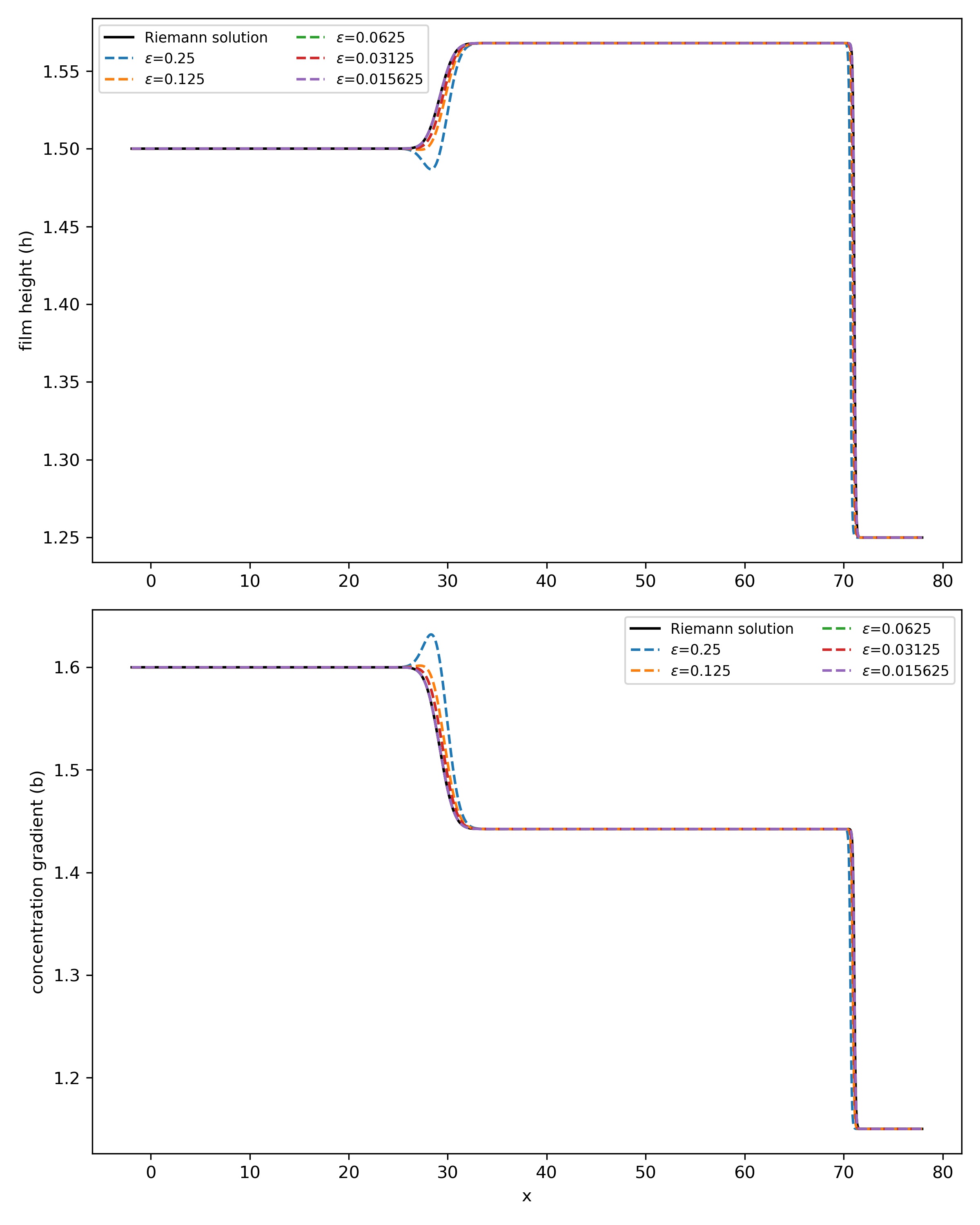}
    \caption{Long-time behaviour of the perturbed Riemann solutions with initial data \eqref{ex:J+S_per} at $t=15.0$.}
    \label{fig:J+S_per_t=15}
\end{figure}

We use the Godunov solver again, and plot the numerical solutions at time $t=1.0$ and $t=15.0$ in Figure \ref{fig:J+S_per_t=1} and Figure \ref{fig:J+S_per_t=15}, respectively. It can be observed that at time $t=1.0$, the solution of the perturbed Riemann problem is $J_1+S_1$ and $J_2+S_2$, which converges to the solution of the Riemann problem as $\varepsilon\rightarrow 0$. Also, the large time behaviour of the perturbed Riemann problem is governed by the solution of the underlying Riemann problem only.
 \end{example}
 \begin{example}
 In this case, we consider the initial data of the form  
 \begin{align}\label{ex:deltaS}
    (h, b)^\top (x, 0)=\begin{cases}
         (1.24, 0.90)^\top,~~~~\qquad -\infty <x<-\varepsilon,\\
        (0.00001, 5.5)^\top, ~~\qquad -\varepsilon<x<\varepsilon, \\
        (1.5, 1.56)^\top, ~~\quad\qquad \varepsilon<x<+\infty.
    \end{cases}
\end{align}
The solution structure for this case corresponds to the interaction of $\delta S_1$ and $J_2+R_2$. We plot the numerical solution at time $t=1.0$ and $t=35.0$ in Figure \ref{fig:deltaS_per_t=1} and Figure \ref{fig:deltaS_per_t=35}, respectively, using the Godunov solver, which again verifies the asymptotic behaviour of the perturbed Riemann problem. 
\begin{figure}
    \centering
    \includegraphics[width=\linewidth]{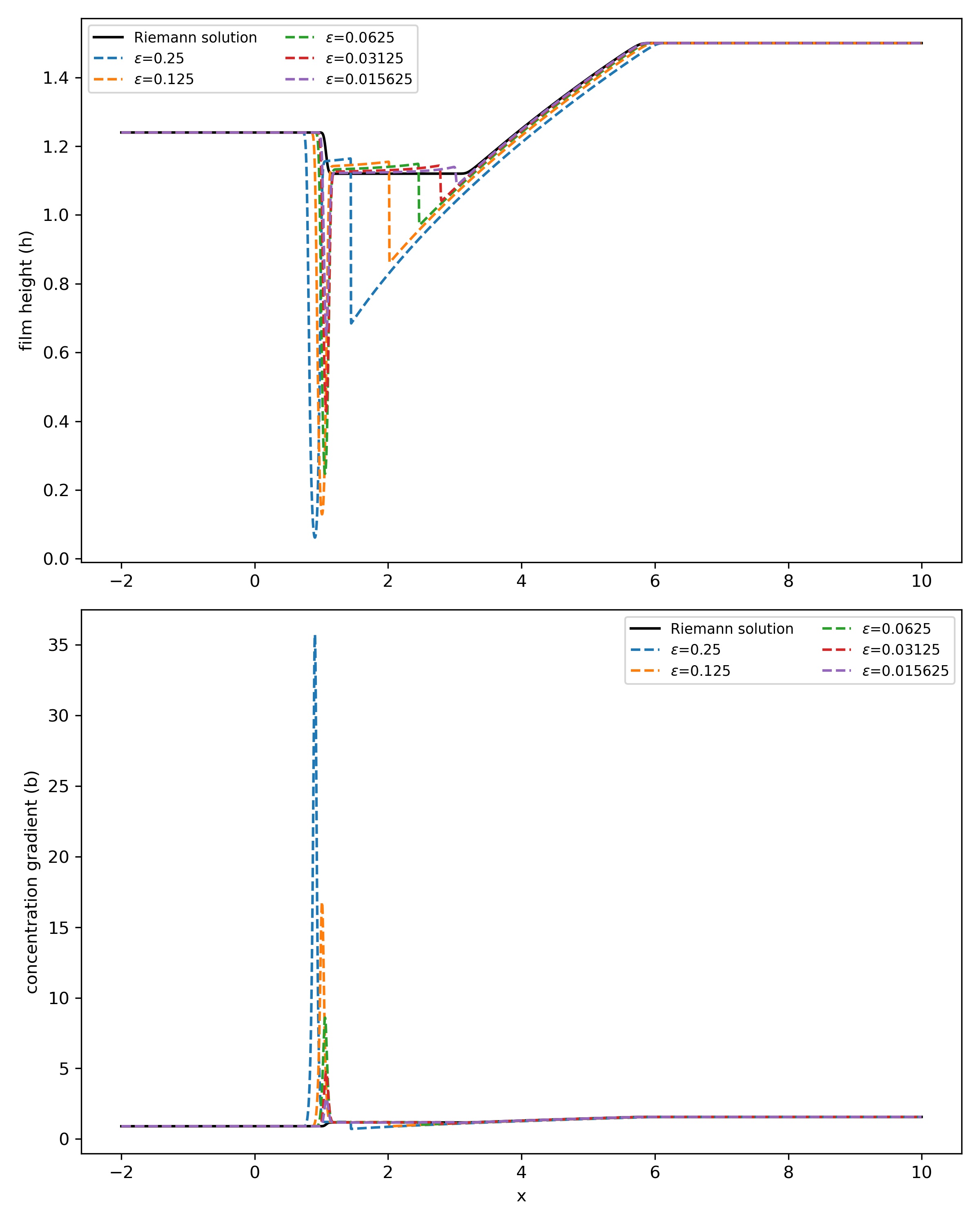}
    \caption{Solution of perturbed Riemann problem with initial data \eqref{ex:deltaS} at $t=1.0$\\
    (Interaction of $\delta S_1$ and $J_2+R_2$).}
    \label{fig:deltaS_per_t=1}
\end{figure}
\begin{figure}
    \centering
    \includegraphics[width=\linewidth]{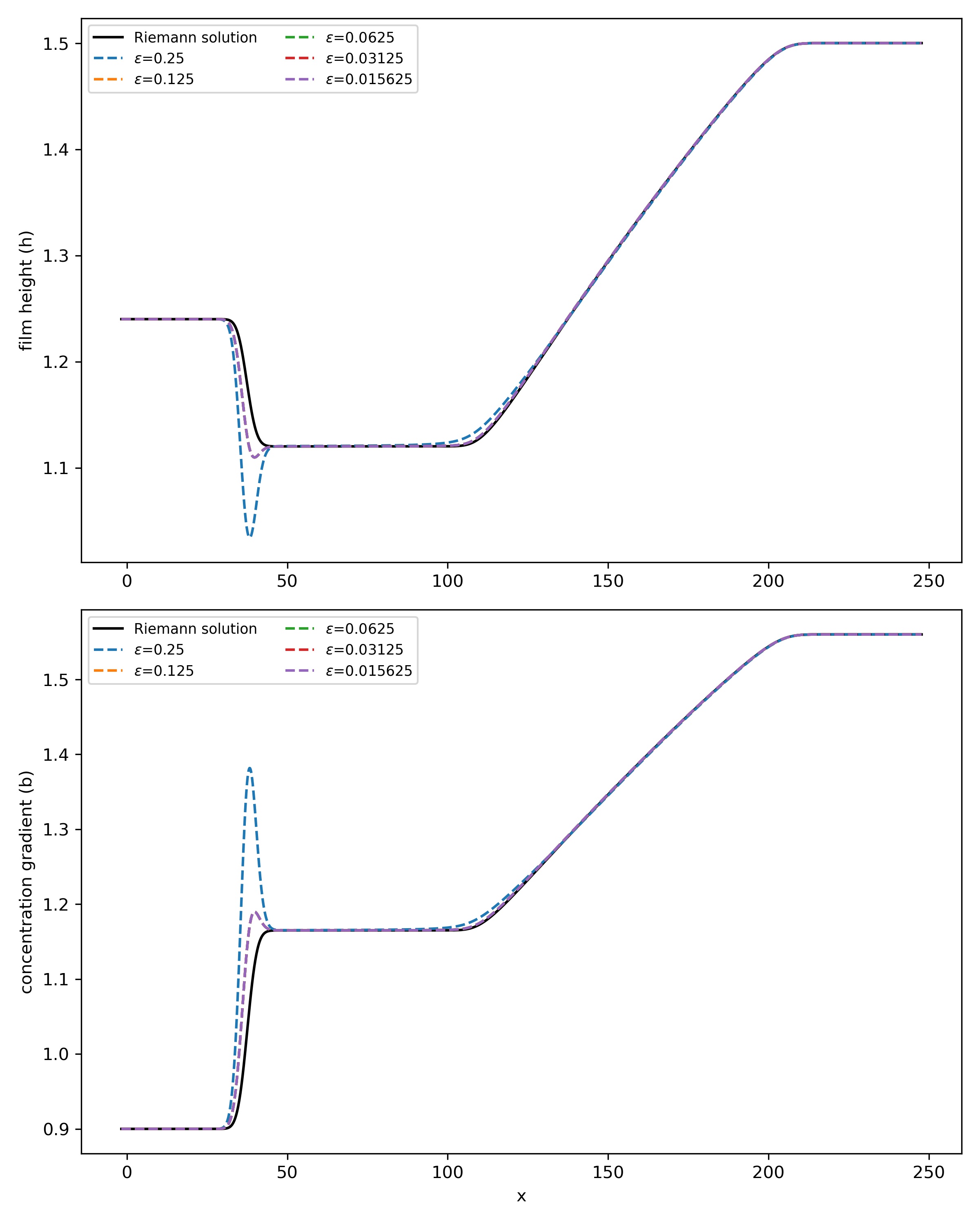}
    \caption{Long time behaviour of the solutions with initial data \eqref{ex:deltaS} at $t=35.0$.}
    \label{fig:deltaS_per_t=35}
\end{figure}
 \end{example}

\section{Conclusions and future outlook}\label{sec: 5}
In this article, we analyzed the Riemann problem for a newly developed $(2\times 2)$ hyperbolic system, which governs the dynamics of thin film flow under the influence of gravity and solute transport. We were able to provide a novel set of classes of entropy-entropy pairs with sufficient conditions for strictly convex entropies. Also, we prove the existence and uniqueness of the Riemann problem and analyzed the stability of the constructed solutions via flux perturbation and perturbation in the initial data. 

A natural next step would be to prove the existence of global $L^{\infty}$ solutions for the system using the compensated compactness approach, where the rich entropy structure of the system can play a big role. However, it will not be straightforward to obtain global bounds due to the nonconvexity of the first Riemann invariant and may need flux perturbation methods to obtain a bounded invariant region and thus induced global bounds. Also, it would be interesting to analyze whether the one-layer case could be extended to its multi-layer counterpart while maintaining the rich mathematical structure of the underlying dynamics.\\\\ 
\textbf{Acknowledgements}~  Financial support by the German Research Foundation (DFG), within the Priority Programme - SPP 2410 Hyperbolic Balance Laws in Fluid Mechanics: Complexity, Scales, Randomness (CoScaRa) is greatly acknowledged.\\\\
\textbf{CRediT authorship contribution statement}~\\
{\bf{Rahul Barthwal:}} Conceptualisation,  Investigation, Methodology, Formal analysis, Validation, Visualisation, Modelling,  Numerical simulations, Writing—original draft, review \& editing, Project administration, {\bf{Christian Rohde:}} Funding acquisition, Writing – review \&
editing, {\bf{Anupam Sen:}}  Investigation, Writing - review \& editing.

\bibliographystyle{abbrv}
\bibliography{citation}
\end{document}